\documentclass[a4paper,11pt]{article}

\usepackage[sumlimits,intlimits]{amsmath}
\usepackage{amsfonts,amsthm,amssymb,mathtools,graphicx,holtpolt,subfigure,bbm,
            hyperref}
\usepackage{geometry}
\usepackage[nofancy]{svninfo}
\usepackage{appendix}
\numberwithin{equation}{section}

\theoremstyle{plain}
\newtheorem{theorem}{Theorem}[section]
\newtheorem{lemma}[theorem]{Lemma}
\theoremstyle{definition}
\newtheorem{algorithm}[theorem]{Algorithm}

\theoremstyle{definition}% remark
\newtheorem{example}[theorem]{Example}

\newcommand{\im}{\text{\normalfont{i}}}
\newcommand{\field}[1]{\mathbbm{#1}}
\newcommand{\fC}{\field{C}}% symbol of set of complex numbers C
\newcommand{\fR}{\field{R}}% symbol of set of real numbers R
\newcommand{\fN}{\field{N}}% symbol of set of natural numbers N={1,2,...}
\newcommand{\abs}[1]{\left\lvert#1\right\rvert}
\newcommand{\acc}{\mathtt{acc}}
\newcommand{\bigO}[1]{\mathcal{O}\!\left(#1\right)}
\newcommand{\bigo}[1]{\mathcal{O}(#1)}

\newcommand{\Order}[1]{\bigO{n^{#1}}}
\newcommand{\ftime}[1]{\textit{#1}}
\newcommand{\ds}{\displaystyle}
\newcommand{\tablesize}{\footnotesize}

\DeclareMathOperator{\sgn}{\mathrm{sgn}}

\newcommand{\abpolter}[2]{\polter{a_{#1}}{b_{#2}}}
\newcommand{\abppolter}[2]{\polter{a_{#1}'}{b_{#2}'}}

\DeclareMathOperator*{\Ket}{\raisebox{-2pt}{\mbox{\Large $\mathbf K$}}}
\newcommand{\cfsum}[2]{\Ket_{#1}^{#2}}
\newcommand{\cftwosum}[3]{\Ket_{#1}^{#2}\left[\,#3\,\right]}
\newcommand{\cfinf}[1][n=1]{\Ket_{#1}^{\infty}}
\newcommand{\K}{\text{\textbf K}}
\newcommand{\mC}{\mathcal{C}}
\newcommand{\mD}{\mathcal{D}}
\newcommand{\wmD}{\widetilde{\mD}}
\newcommand{\mDe}{\mathcal{D}^{=}}
\newcommand{\mDn}{\mathcal{D}^{\ne}}
\title{On the convergence acceleration of some continued fractions%
}
\author{%
	Rafa\l{} Nowak%
	\footnote{Institute of Computer Science, University of Wroc\l{}aw, ul. Joliot Curie 15, 50-383 Wroc\l{}aw, Poland,\newline
	e-mail: \href{mailto:rno@cs.uni.wroc.pl}{\texttt{rno@cs.uni.wroc.pl}}}%
}

\date{%
 	\svnInfoLongDate%
 	\footnote{Rev.\ \svnInfoRevision, \svnInfoLongDate, \svnInfoTime}%
}

\begin{document}
 \svnInfo $Id: convacc-tvcf.tex 2097 2012-03-04 06:26:48Z rno $
\maketitle
\begin{abstract}
	A~well known method for convergence acceleration of a~continued fraction $\K(a_n/b_n)$ is based on the use of modified approximants $S_n(\omega_n)$ in place of the classical ones $S_n(0)$, where $\omega_n$ are close to tails $f^{(n)}$ of the continued fraction. Recently (Numer. Algorithms 41 (2006), 297--317), the author proposed an~iterative method producing tail approximations whose asymptotic expansion's accuracy is improved in each step.
	This method can be applied to continued fractions $\K(a_n/b_n)$, where, for sufficiently large $n$, $a_n$ and $b_n$ are polynomials in $n$ ($\deg a_n=2$, $\deg b_n\leq 1$).
	The purpose of this paper is to extend this idea to the class of continued fractions $\K(a_n/b_n + a_n'/b_n')$, where $a_n$, $a_n'$, $b_n$, $b_n'$ are polynomials in $n$ ($\deg a_n=\deg a_n', \deg b_n=\deg b_n'$).
	We~give examples involving continued fraction expansions of some mathematical constants, as well as elementary and special functions.
\end{abstract}

%BEGINOFTEXT - do not remove this comment
\section{Introduction}
Many valued mathematical constants and elementary or special functions have continued fraction expansions. For example, the following continued fraction (CF, in short)
\begin{equation}
	\label{E:one}
	\arctan x = \polter{x}{1} + \cfsum{n=1}{\infty}\polter{(2n-1)^2x^2}{2n+1-(2n-1)x^2}
	\qquad (-1\leq x\leq 1)
\end{equation}
is very useful expansion of the well-known inverse trigonometric function. If, however, \mbox{$x\approx 1$}, 
then it converges very slowly and application of some methods of convergence acceleration is required (see \cite[p. 19, Eq.~(11)]{Perron}).
If $x\neq 1$, then CF~\eqref{E:one} can be transformed equivalently into limit $1$-periodic CF of loxodromic type, and one can use the \ftime{improvement machine} --- the method established by Lorentzen and Waadeland (see \cite{Lorentzen95}, \cite[p.~227]{Lorentzen08}). But if $x=1$, then~\eqref{E:one} is of elliptic type and, according to \cite{Lorentzen95}, one can use its even part and transform it to the CF of parabolic type, of the form $\K(c_n/1)$, where $c_n+\tfrac14=\tfrac{3}{64}n^{-4}+\bigO{n^{-5}}$. Since $c_n+\tfrac{1}{4}=\bigO{n^{-4}}$, one can use modified approximants $S_n(-1/2)$ --- the \textit{fixed point method} --- in order to accelerate the convergence of~$\K(c_n/1)$ (see \cite[p.~219]{Lorentzen08}). However, if $c_n+\tfrac{1}{4}=\bigO{n^{-2}}$, then such a~method usually does not succeed. One can verify that CF~\eqref{E:one} can be effectively computed using the method introduced in \cite{Nowak06}, where the convergence acceleration of $\text{CF}\in \mC_{20}\cup\mC_{21}$ was proposed; $\
mC_{kl}$ means the class of CFs of the form $\K(a_n/b_n)$ with $a_n$ and $b_n$ being polynomials in $n$ ($\deg a_n=k$, $\deg b_n=l$). In~this paper, we extend this result to the class $\mD_{kl}$ of so-called \ftime{two-variant} CFs (see~\cite{Paszkowski03}) of the form
\begin{equation}\label{etvcfrac}
	b_0'+\cftwosum{n=0}{\infty}{\abpolter nn+\abppolter nn},
\end{equation}
where $a_{n}, a_{n}'$ and $b_{n}, b_{n}'$ are polynomials in $n$, such that
$\deg a_{n}=\deg a_{n}'=k$, $\deg b_{n}=\deg b_{n}'=l$.
One can find a~lot of examples of such CFs in the Perron's book~\cite{Perron}. More CF expansions can be found in (chronological order): \cite{Wall48}, \cite{Lorentzen08} and \cite{Cuyt08}. The authors of \cite{Cuyt08} have developed the website which shows the applications of continued fractions to approximation of many special functions (see~\url{http://www.cfhblive.ua.ac.be/}).

For example, the following CF expansion holds
\begin{equation}\label{E:cn}
	\int_{0}^{\infty}e^{-tx}\mathrm{cn}(t;\,k)\,dt= \polter{1}{x}+\cftwosum{n=1}{\infty}{\polter{(2n-1)^2}{x}+\polter{(2n)^2 k^2}{x}}
	\qquad (0<k<1,\; x>0),
\end{equation}
where $\mathrm{cn}(t;\,k)$ is the Jacobi elliptic function, very often applied in electro- and magnetostatics, fluid dynamics, as well as nonlinear integrable equations (see~\cite[p.~220, Eq.~(15)]{Perron} or \cite[p.~283]{Lorentzen08}).
Although the equation \eqref{E:cn} is known for almost 100 years (the first edition of Perron's book \cite{Perron} was issued in 1913), current computer algebra systems do not use it. For instance, the Maple system is not able to compute the integral in equation \eqref{E:cn} with $x = 0.8$ and $k = 0.9$; if we replace the upper limit $\infty$ by $500$, then the correct result is obtained after $1.5$ minutes of waiting.  Thus, the problems which require computing a~lot of such integrals are almost insolvable using the current mathematical software like Maple or Mathematica. The motivation for using aforementioned techniques of convergence acceleration is, among others, the fact that CF expansion \eqref{E:cn} allows us to compute such integrals in a split of a~second.

Our main result is a~new method for convergence acceleration of two-variant CFs.
Below, we give some notation used in this paper.
In Section~\ref{s:main}, we describe the class $\mD$ of considered CFs (see~\eqref{D}).
We analyse tails of CFs and try to find out their asymptotic form. Our investigations require to study many cases of CFs' subclasses. For the sake of clarity, we moved some technical results to \ref{apa}.
We show that derivation of the coefficients of~asymptotic expansion of tails is quite costly.
Thus, we propose an~iterative method that improves the accuracy of tails'
approximations.
In Section~\ref{s:method}, we describe our method in detail. The main result is in Theorem~\ref{T:convacc}.
Next, in Section~\ref{s:alg} we use this result repeatedly in order to obtain the algorithm, which is of iterative character, producing tail approximations whose asymptotic expansion's accuracy is improved in each step. This idea is summarized in the Algorithm~\ref{alg}.
Finally, in Section~\ref{s:num}, we report the results of experiments involving some CFs belonging to the class $\mD$.

\subsection*{Some notation}
Let $S_n(\omega)$ denote the \ftime{$n$-th modified approximant} of CF~\eqref{etvcfrac}; see \cite[pp.~7, 56]{Lorentzen92} or~\cite[(1.2.1), p.~5]{Lorentzen08} for more details.
Clearly, we have
\begin{align*}
	S_0(\omega)&=b_0'+\omega,\\%
	S_{2n-1}(\omega) &= b_0'+\cftwosum{k=1}{n-1}{\polter{a_k}{b_k}+\polter{a_k'}{b_k'}}+\polter{a_{n}}{b_n+\omega},& &n\in\fN,\\
	S_{2n}(\omega)   &= b_0'+\cftwosum{k=1}{n-1}{\polter{a_k}{b_k}+\polter{a_k'}{b_k'}}+\polter{a_{n}}{b_n}+\polter{a_{n}'}{b_n'+\omega},& &n\in\fN.
\end{align*}
We use the symbol $f^{(n)}$ to denote the \ftime{$n$-th tail}, i.e.,
\begin{align*}
	f^{(2n-1)} &= \polter{a_n'}{b_n'} + \cftwosum{k=n+1}{\infty}{\polter{a_k}{b_k}+\polter{a_k'}{b_k'}},& &n\in\fN,\\
	f^{(2n)}   &=                       \cftwosum{k=n+1}{\infty}{\polter{a_k}{b_k}+\polter{a_k'}{b_k'}},& &n\in\fN\cup\{0\}.
\end{align*}

\section{Asymptotics of tails}\label{s:main}
Our method for convergence acceleration is based on the obvious fact: if CF is
convergent to $V$, then all its tails are also convergent, and
\begin{equation}\label{ematail}
	S_{n}(f^{(n)}) = V.
\end{equation}
Since $S_{n}(\omega)$ is the linear fractional transformation in $\omega$, it is of great importance to find the best approximation $\omega_n$ of the tail $f^{(n)}$. Then the modified approximant $S_n(\omega_n)$ may be a~very good approximation of $V$.

In Subsection~\ref{s:class}, we describe the class $\mD$ of considered CFs.
Taking into account their `two-variant' character (cf.~\eqref{etvcfrac}), we deal with odd tails only, \[ u_n\coloneqq f^{(2n-1)}.\]
One can verify that tails $u_n$ of CF~\eqref{etvcfrac} satisfy the following recurrence relation:
\begin{equation}\label{urec}
	u_n=\cfrac{a_n'}{b_n'+\cfrac{a_{n+1}}{b_{n+1}+u_{n+1}}} \qquad (n=1,2,\ldots).
\end{equation}
In Subsection~\ref{astail}, we analyse the behaviour of tails $u_n$ and state their asymptotic expansion (see~\eqref{eunexp}). Also, we give formulae for its leading coefficients (see Theorem~\ref{tmu}). Let us remark that derivation of them requires considering all the subclasses in $\mD$, separately. Thus, we move part of calculations to \ref{apa}.

\subsection{Class $\mD$}\label{s:class}
We consider some subclasses of CFs of the form \eqref{etvcfrac}, where
\begin{gather}\label{easab}
\begin{aligned}
	a_{n+1} &= p_{-2}n^2+p_{-1}n+p_0+\ldots,\qquad   & a_{n}' &= p_{-2}'n^2+p_{-1}'n+p_0'+\ldots,\\
	b_{n+1} &= q_{-2}n^2+q_{-1}n+q_0+\ldots,   & b_{n}' &= q_{-2}'n^2+q_{-1}'n+q_0+\ldots.
\end{aligned}
\end{gather}
We say that CF \eqref{etvcfrac} is of class $\mD_{kl}$ if $p_k, p_k', q_l, q_l'$ are leading coefficients in the above expansions. In this paper, we consider only $k\in\{1,2\}$ and $l\in\{0,1\}$, since any CF of class $\mD_{k+2, l+1}$ is equivalent to some of the class $\mD_{kl}$.
%TODO (see equivalent transformation).

Let us fix some notation related to CFs' subclasses considered here.
In this paper, we consider the CFs belonging to
\begin{equation}\label{D}
	\mD\coloneqq \mDe_{10}\cup\mDn_{10}\cup\mD_{11}\cup\mDe_{20}\cup\mDn_{20}\cup\wmD_{21};
\end{equation}
see Table~\ref{Tab:Rozwazane_podklasy}, where the characterization of all the subclasses is~given.
\begin{table}[htb]
\begin{center}
	\caption{Considered subclasses of CFs}
	\label{Tab:Rozwazane_podklasy}
	\begin{tabular}{c|c}
		Subclass symbol & Characterization \\\hline
		$\mDe_{10}\subset\mD_{10}$ &
		\parbox[m]{.6\textwidth}{\vskip -8pt%
		\begin{gather*}
			p_{-1}=p'_{-1}
			\qquad\text{and}\qquad
			\frac{q_0\,q'_0}{p_{-1}} \notin \left(-\infty,0\right]
		\end{gather*}\vskip -5pt%
		} \\\hline
		$\mDn_{10}\subset\mD_{10}$ &
		\parbox[m]{.6\textwidth}{\vskip -8pt%
		\begin{gather*}
			\abs{p_{-1}'} \neq\abs{p_{-1}}
		\end{gather*}\vskip -5pt%
		} \\\hline
		$\mDe_{20}\subset\mD_{20}$ &
		\parbox[m]{.7\textwidth}{%\vskip -8pt%
		\[
			%\label{E:De_20:zalozenie}
			p_{-2}=p'_{-2}
			\quad\text{and}\quad
			\frac{\beta^2-4\alpha\gamma}{p_{-2}^2} \notin\left(-\infty,0\right],
		\]
		where $\alpha=q'_{0},\;\beta=p_{-1}-p'_{-1}-p_{-2},\;\gamma=-p_{-2}\,q_{0}$
		}
		\\\hline
		$\mDn_{20}\subset\mD_{20}$ &
		\parbox[m]{.6\textwidth}{\vskip -8pt%
		\begin{gather*}
			\abs{p_{-2}'} \neq\abs{p_{-2}}
		\end{gather*}\vskip -5pt%
		} \\\hline
		$\wmD_{21}\subset\mD_{21}$  &
% 		\parbox[m]{.6\textwidth}{\vskip 5pt%
		The roots $x_0$, $x_1$ of equation
		\(
			%\label{E:D21:tau-2:rownanie}
			x = \frac{p_{-2}'}{q_{-1}'+\tfrac{p_{-2}}{q_{-1}+x}}
		\)
		satisfy
		\(
			%\label{E:D21:zalozenie}
			\abs{\tfrac{p_{-2}'}{q_{-1}'}-x_0} \neq \abs{\tfrac{p_{-2}'}{q_{-1}'}-x_1}
		\)
% 		}
	\end{tabular}
\end{center}
\end{table}

\subsection{Asymptotic expansion of tails}\label{astail}
It follows from~\eqref{urec} that tails $u_n$ satisfy the equation
(see~also~\cite[(1.6)]{Paszkowski03}):
\begin{equation}\label{oddtailsrel}
	\left(b_n' b_{n+1} + a_{n+1}\right)  X_n + b_n' X_n X_{n+1}  - a_n' X_{n+1} - a_n' b_{n+1} = 0.
\end{equation}
Now, we analyse the asymptotic behaviour of its solutions. 
Observe that recurrence relation~\eqref{oddtailsrel} is bilinear and all its
coefficients have the asymptotic expansion in powers of~$n^{-1}$
(cf.~\eqref{easab}).
Therefore, its solutions have the asymptotic expansion in
powers f $n^{-1/2}$. Moreover, this expansion is at most $\bigo{n^{2}}$:
\begin{equation}\label{eXnexp}
	X_n = \tau_{-4} \, n^2 + \tau_{-3}\, n^{3/2} + \tau_{-2}\, n + \tau_{-1}\, n^{1/2} + \tau_0 + \tau_1\, n^{-1/2} + \bigo{n^{-1}}.
\end{equation}
%for some coefficients $\tau_j$.
More precisely, equation \eqref{oddtailsrel} has two solutions $X_n', X_n''$ of
the form \eqref{eXnexp}. According to Paszkowski~\cite{Paszkowski03}, the
sequence $u_n$ is a~solution $X_n\in\{X_n', X_n''\}$, such that $\{X_{n}
a_{n+1} / (b_{n+1} + X_{n+1})\}$ is asymptotically smaller in~modulus.
It means that the problem in solving the equation~\eqref{oddtailsrel} reduces
to choosing suitable coefficients $\tau_j$, such that
\begin{equation}\label{eunexp}
	u_n = \tau_{-\mu} \, n^{\mu/2} + \tau_{1-\mu}\, n^{\mu/2-1/2} + \tau_{2-\mu}\, n^{\mu/2-1} + \tau_{3-\mu}\, n^{\mu/2-3/2} + \ldots.
\end{equation}
Of course, knowing the coefficients $\tau_j$, we can give good approximations of the tails $u_{n}$, and good approximations of CF value, as well (cf.~\eqref{ematail}).
In this Subsection, we give formulae for some portion of the coefficients
$\tau_{-\mu}, \tau_{1-\mu}, \ldots$, in the case of any subclass in $\mD$, named
in Subsection~\ref{s:class}.

First, we put the value $\mu\leq 4$ (cf.~\eqref{eXnexp} and~\eqref{eunexp}).
From~\eqref{easab}, it follows easily that $\mu\leq 2$ if
$\text{CF}\notin\mD_{20}$.
In the sequel, we allow to be $\tau_{-\mu}=0$. We will show that this may happen only in the case of $\mDn_{10}$ or $\mD_{11}$.
%TODO: sprawdzić to
Therefore, we call $\tau_{-\mu}$ the \ftime{beginning coefficient} (not the leading one) of expansion \eqref{eunexp}.
According to Lemma \ref{ld1020}, we choose
\begin{equation*}%\label{E:mu}
	\mu\coloneqq
	\begin{cases}
		1, & \text{if $\text{CF}\in\mDe_{10}$},\\
		4, & \text{if $\text{CF}\in\mDe_{20}$},\\
		2 & \text{otherwise,}
	\end{cases}
\end{equation*}
and assume that $\tau_{-\nu}=0$ if $\nu>\mu$. The proof of Lemma~\ref{ld1020} is rather technical and requires considering all the subclasses in~$\mD$ separately.

Second, let us look at the beginning coefficient $\tau_{-\mu}$. Our main
results are as follows. Lemma \ref{lmueq} shows that $\tau_{-\mu}$ satisfies
quadratic equation $\alpha\tau^2+\beta\tau+\gamma=0$, while Theorem~\ref{tmu}
specifies which root it is, where $\alpha$, $\beta$ and $\gamma$ are given
in~Table~\ref{tab:alpha}.
Also, we give there simple formulae for $\tau_{-\mu}$; in Lemmata~\ref{ld1020},
\ref{ld20} one can find the formulae for the next coefficients of the
expansion~\eqref{eunexp}.
\begin{table}[htb]
\centering
% \begin{center}
% 	\extrarowheight 5pt
	\caption{Coefficients $\alpha, \beta, \gamma$}
	\label{tab:alpha}
	\begin{tabular}{c|ccc}
		Subclass & $\alpha$ & $\beta$ & $\gamma$ \\\hline
		$\mDe_{10}$ & $q_0'$ & $0$              & $-q_0\, p_{-1}$ \\
		$\mDn_{10}$ & $q_0'$ & $p_{-1}-p'_{-1}$ & $0$ \\
		$\mD_{11}$  & $q_{-1}'$ & $q_{-1}\,q'_{-1}$ & $0$ \\
		$\mDe_{20}$ & $q'_0$ & $p_{-1}-p'_{-1}-p_{-2}$ & $-q_0\, p_{-2}$ \\
		$\mDn_{20}$ & $q_0'$ & $p_{-2}-p'_{-2}$ & $0$ \\
		$\wmD_{21}$ & $q_{-1}'$ & $p_{-2}-p'_{-2}+q_{-1}\,q'_{-1}$ & $ -q_{-1}\,p'_{-2}$
	\end{tabular}
% \end{center}
\end{table}
%The Lemma~\ref{ld20} is an addition to Lemma~\ref{ld1020} and results from Theorem \ref{tmu}. However the proof of Theorem~\ref{tmu} needs only the Lemma~\ref{ld1020} and the following result.
\newpage
\begin{theorem}\label{tmu}
	The beginning coefficient $\tau_{-\mu}$ is given explicitly in the following way.
	\begin{enumerate}
	\item In the case of $\mDe_{10}$ $(\mu=1)$:
		\begin{equation}\label{E:De10:tau_-1}
			\tau_{-1}\coloneqq\sgn\left(\Re\frac{q'_0\sqrt{\frac{q_0\,p_{-1}}{q'_0}}}{p_{-1}}\right) \sqrt{\frac{q_0\,p_{-1}}{q'_0}};
		\end{equation}
	\item In the case of $\mDn_{10}$ $(\mu=2)$:
		\begin{equation*}%\label{E:Dn10:tau_-2}
			\tau_{-2}\coloneqq\begin{dcases} 0 & \text{if $\abs{p_{-1}'}<\abs{p_{-1}}$}, \\ \frac{p'_{-1}-p_{-1}}{q_0'} & \text{if $\abs{p_{-1}'}>\abs{p_{-1}}$}; \end{dcases}
		\end{equation*}
	\item In the case of $\mD_{11}$ $(\mu=2)$:
		\begin{equation}
			\label{E:D11:tau_-2}
			\tau_{-2} \coloneqq 0;
		\end{equation}
	\item In the case of $\mDe_{20}$ $(\mu=2)$:
		\begin{equation}
			\label{E:De20:tau_-2}
			\tau_{-2}\coloneqq\dfrac{-\beta+\sgn\left(\Re\frac{\sqrt{\beta^2-4\alpha\gamma}}{p_{-2}}\right)\sqrt{\beta^2-4\alpha\gamma}}{2\alpha},
		\end{equation}
		where $\alpha,\beta$ and $\gamma$ are given in Table~\ref{tab:alpha};%$\alpha=q_0'$, $\beta=p_{-1}-p_{-1}'-p_{-2}$, $\gamma=-q_0\, p_{-2}$;
	\item In the case of $\mDn_{20}$ $(\mu=4)$:
		\begin{equation*}
% 			\label{E:Dn20:tau_-4}
			\tau_{-4}\coloneqq\begin{dcases} 0 & \text{if $\abs{p_{-2}'}<\abs{p_{-2}}$}, \\ \frac{p'_{-2}-p_{-2}}{q_0'} & \text{if $\abs{p_{-2}'}>\abs{p_{-2}}$}; \end{dcases}
		\end{equation*}
	\item\label{item:D21:tau}
		In the case of $\wmD_{21}$  $(\mu=2)$: we choose $\tau_{-2}$
		as the first of the roots $\{x_0, x_1\}$ of equation \eqref{etaueq} if
		\begin{equation}
			\label{E:D21:tau_-2}
			\abs{\frac{p_{-2}'}{q_{-1}'}-x_0} < \abs{\frac{p_{-2}'}{q_{-1}'}-x_1}.
		\end{equation}
	\end{enumerate}
\end{theorem}
\begin{proof}
According to Paszkowski~\cite{Paszkowski03}, the sequence of odd tails
$f^{(2n-1)}$ of CF \eqref{etvcfrac} is such a~solution of recurrence equation
\eqref{oddtailsrel}, for which the product $f^{(2n-1)} f^{(2n)}$ is
asymptotically smaller in~modulus.

Before we start, let us observe that setting \[I_n \coloneqq f^{(2n-1)}
f^{(2n)}=\frac{a_{n+1}\,u_n}{b_{n+1}+u_{n+1}}\] yields, after some simple
calculations that use only Lemmata~\ref{ld1020} and~\ref{lmueq}:
	\begin{enumerate}
	\item\label{it:De10} In the case of $\mDe_{10}$:
		\begin{equation}
 			\label{E:De10:iloczyn}
			%t_{2n-1} t_{2n} \sim 
			I_n= p_{-1}\,n - \frac{p_{-1} q_0}{\tau_{-1}} \sqrt{n} + \bigO{1};
		\end{equation}%
	\item\label{it:Dn10} In the case of $\mDn_{10}$:
		\begin{equation}\label{E:Dn10:iloczyn}
			I_n=
			\begin{dcases}
				p_{-1}'\,n+\bigO{1} & \text{if $\tau_{-2}=0$,}\\
				p_{-1} \,n+\bigO{1} & \text{otherwise;}
			\end{dcases}
		\end{equation}
	\item\label{it:D11} In the case of $\mD_{11}$:
		\begin{equation}\label{E:D11:iloczyn}
		I_n=
		\begin{dcases}
			\frac{p_{-1}'\, p_{-1}}{q_{-1}'\, q_{-1}} + \bigo{n^{-1/2}} & \text{if $\tau_{-2}=0$,}\\
			q_{-1}'\,q_{-1}\,n^2+\bigO{n}                               & \text{otherwise;}
		\end{dcases}
		\end{equation}
	\item\label{it:De20} In the case of  $\mDe_{20}$:
		\begin{equation}\label{E:De20:iloczyn}
			%t_{2n-1} t_{2n}= 
			I_n = p_{-2}\, n^2 + \left(p_{-1}-p_{-2}-\frac{p_{-2}\,q_0}{\tau_{-2}}\right) n + \bigO{1};
		\end{equation}%
	\item\label{it:Dn20} In the case of  $\mDn_{20}$:
		\begin{equation}\label{E:Dn20:iloczyn}
			I_n =
			\begin{dcases}
			p_{-2}'\,n^2 +\bigO{n} & \text{if $\tau_{-4}=0$,}\\
			p_{-2} \,n^2 +\bigO{n} & \text{otherwise;}
			\end{dcases}
		\end{equation}
	\item\label{it:D21} In the case of  $\wmD_{21}$
		\begin{equation}\label{E:D21:iloczyn}
				I_n = \frac{\tau_{-2}\,p_{-2}}{\tau_{-2}+q_{-1}}\,n^2 +\bigo{n^{3/2}}.
		\end{equation}
	\end{enumerate}
Let us observe that the last formula is well-defined. Otherwise, assuming
$\tau_{-2}=-q_{-1}$ and using \eqref{E:D21:m=-6}, we derive $p_{-2}q_{-1}=0$,
which is a~contradiction.
The formulae \eqref{E:Dn10:iloczyn}, \eqref{E:D11:iloczyn} and~\eqref{E:Dn20:iloczyn} easily imply the parts 2, 3 and 5 of the thesis.
To complete the proof of the parts 1 and 4, it is enough to make use of the two
obvious facts.
First, if $z_n=an^\kappa+bn^\lambda+\bigO{n^{\nu}}$ for any $\kappa>\lambda>\nu$, then
\[
	\abs{z_n}^2 = \abs a n^{2\kappa} + 2(\Re a\Re b+\Im a\Im b)n^{\kappa+\lambda}+\bigo{n^{\max\{2\lambda,\kappa+\nu\}}}.
\]
Second, $\sgn \Re \frac ab = \sgn (\Re a\Re b+\Im a\Im b)$ holds for any $a, b\in\fC$ $(b\neq0)$.
In order to complete the proof of the part 6, let us observe that
inequality~\eqref{E:D21:tau_-2} is~equivalent to
\[
	\abs{x_0\, q_{-1}' - p_{-2}'}<\abs{x_1\, q_{-1}' - p_{-2}'}.
\]
By multiplying both sides by $|{p_{-2}\,q_{-1}}/{q_{-1}'}|$, we obtain
	\[
		\abs{-\frac{q_{-1}\,p_{-2}'}{q_{-1}'}p_{-2} + x_0\,p_{-2}\,q_{-1}}<\abs{-\frac{q_{-1}\,p_{-2}'}{q_{-1}'}p_{-2} + x_1\,p_{-2}\,q_{-1}}.
	\]
Now, let us note that $x_0 x_1 = -\frac{q_{-1}\,p_{-2}'}{q_{-1}'}$
(equal to the~quotient $\gamma/\alpha$ in the case $\wmD_{21}$; see
Table~\ref{tab:alpha}).
Hence, we have
	\(
		\abs{x_0\,x_1\,p_{-2} + x_0\,q_{-1}\,p_{-2}} < \abs{x_0\,x_1\,p_{-2} + x_1\,q_{-1}\,p_{-2}},
	\)
and thus
	\[
		\abs{\frac{x_0\,p_{-2}}{x_0+q_{-1}}} < \abs{\frac{x_1\,p_{-2}}{x_1+q_{-1}}}.
	\]
Therefore, the choice $\tau_{-2}\coloneqq x_0$ implies that $I_n$ has
asymptotically smaller absolute value (cf.~\eqref{E:D21:iloczyn}).
\end{proof}
% Let us notice that Lemma~\ref{ld1020} does not consider the case of subclass $\mDe_{20}$.
% Now, such results can be established by making use of equation~\eqref{E:De20:tau_-2}.
% We state them in Lemma~\ref{ld20} in \ref{apa}.

\section{Method}\label{s:method}
In this Section, we present the method for convergence acceleration.
We consider only CFs of subclasses in~$\mD$, given in~\eqref{D}. However, in
general, our idea can be applied to any CF whose sequence of odd (or even) tails
$u_n$ has asymptotic expansion of the form~\eqref{eunexp}. Then the only
problem is to derive the initial approximations of the tails $u_n$.

The main objective is to improve, in each step, the accuracy of approximations
of tails $u_n$. Of course, one can use the technique of Subsection~\ref{astail}
to obtain a~suitable number of coefficients $\tau_j$ in
the~expansion~\eqref{eunexp}. This approach was discovered more than
half-century ago by Wynn~in~\cite{Wynn59}; see also~\cite[p.~232]{Lorentzen08}.
Let us note that in such a~case, each coefficient $\tau_j$ can be expressed in
terms of all the previous coefficients $\tau_{j-1}$, $\tau_{j-2}$, \ldots,
$\tau_{-\mu}$, and hence the computation of them explicitly is quite costly.
For example, in the case of CF$\in\mC_{20}\cup\mC_{21}$, the
recurrence formula for the sequence of tails of CF was given
in~\cite[Thm.~2.6]{Nowak06}.

Our goal is to obtain better and better approximations of the tails
\eqref{eunexp} without computing the coefficients $\tau_j$ explicitly.
We will construct approximations $u_n^{(m)}$ of tails $u_n$, such that
$u_n^{(m)}$ are of increasing accuracy, for consecutive values of $m$, in the
sense that:
\[ (u_n^{(m+1)}-u_n)/(u_n^{(m)}-u_n) \to 0 \qquad \text{when}\qquad
	n\to\infty.
\]
More precisely, if $u_n^{(m)}-u_n=\mathcal{O}({n^{-m/2}})$, then $u_n^{(m+1)}-u_n=\mathcal{O}({n^{-m/2-\theta}})$, where $\theta>0$ is given in~\eqref{theta}.
Our method is of the iterative character, since we express $u_{n}^{(m+1)}$ in terms of $u_{n}^{(m)}$ and $u_{n+1}^{(m)}$.
Similar concept was given by Lorentzen and Waadeland, and called `improvement
machine' (see \cite[p.~227]{Lorentzen08} or~\cite[\S6]{Lorentzen95}, for more
details). However, their method cannot be applied to all the subclasses in $\mD$
and, as numerical examples show, is less efficient than ours.

The following technique is very similar to the one presented in~\cite{Nowak06}, but is more subtle, since we deal with more subclasses of CFs and consider only their odd tails, whose asymptotic expansion can have powers of $n^{-1/2}$.
However, from Lemmata~\ref{ld1020} and~\ref{ld20}, it follows that tails $u_n$ mostly have expansion in $n^{-1}$. Only in the case of $\mDe_{10}$ the expansion~\eqref{eunexp} may contain powers of $n^{-1/2}$. Numerical examples show that it is hard to accelerate such a~class of CFs.

\subsubsection*{Step of iteration}
Now, we describe a~single step of the method in detail.
Let us assume that we have some \ftime{computable} approximations $u_n'$ of tails $u_n$, in the sense that
\begin{equation}
	\label{E:delta_prim}%EEEEEEEEEEEEEEEE
	\delta_n' \coloneqq u_n'-u_n = c \, n^{-m/2} + \mathcal O(n^{-(m+1)/2}),%Thorder
\end{equation}
where $m\geq 0$ is an~integer number and $c\neq 0$ is an unknown constant.
Here, the name \ftime{computable} means that we can compute them for
consecutive values of $n\in\fN$.
The formula~\eqref{E:delta_prim} says that asymptotic expansions on $u_n$ and
$u_n'$ are identical up to the term $n^{-(m-1)/2}$, and so we say that $u_n'$
has the \ftime{order} $m/2$. Since the approximations $u_n'$ are different from
$u_n$, they do not satisfy the equation~\eqref{oddtailsrel}. Thus, let us define
new approximations,
\begin{equation}
	\label{E:u_n^+}%EEEEEEEEEEEEEEEE
	u_n^{+}\coloneqq\cfrac{a_n'}{b_n'+\cfrac{a_{n+1}}{b_{n+1}+u_{n+1}'}}
\end{equation}
(cf.~\eqref{urec}), and observe that they are computable, too. They do not give
better results than the approximations $u_n'$, because one can check that
modified approximants $S_{2n-1}(u_n^+)$, $S_{2n+1}(u_{n+1}')$ are equal.
Nevertheless, we can do a~trick, using both of them in order to obtain new
approximations $u_n''$ of higher order. The main results are summarized in
Theorem \ref{T:convacc}, and the details are in its proof. The trick is the same
as in~\cite{Nowak06}. Roughly speaking, we want to find a~relationship between
$\delta_n'$ and
\begin{equation}
	\delta_n^+ \coloneqq u_n^+ - u_n.
\end{equation}
It can be easily verified that $\delta_n^+$ may be expressed by $\delta_{n+1}'$ as follows
\begin{equation}
	\label{E:delta_n^+}%EEEEEEEEEEEEEEEE
	\delta_n^+ = \psi_n\,\delta_{n+1}',
\end{equation}
where
\begin{equation}
	\label{E:psi_n}
	\psi_n \coloneqq \frac{a'_n a_{n+1}}{\left(a_{n+1}+b'_n\,b_{n+1}+b'_n\,u_{n+1}'\right)\left(a_{n+1}+b'_n\,b_{n+1}+b'_n\,u_{n+1}\right)}.
\end{equation}
However, the formula \eqref{E:delta_n^+} is useless, since $\psi_n$ is expressed
in terms of the tails $u_{n+1}$, which are unknown and rather incomputable.
The major concept is to transform $\delta_{n+1}'$ into $\delta_n'$ in
the relation \eqref{E:delta_n^+} and put $u_{n+1}'$ instead of $u_{n+1}$
in~\eqref{E:psi_n}. This is done in order to obtain the relationship of the
form
\begin{equation}\label{relation}
	\varphi_n' \delta_n^+ - \psi_n' \delta_n' = \bigo{n^{-m/2-\zeta}},
\end{equation}
for some computable functions $\varphi_n'$, $\psi_n'$ and $\zeta>0$.
Then the new approximations $u_n''$ can be found a~solution of~\eqref{relation},
provided that the right-hand side of~\eqref{relation} is replaced by zero and
$u_n$ (stemmed from $\delta_n'$ and $\delta_n^+$) by $u_n''$. Clearly, we
express $u_n''$ in terms of $u_n'$, $u_n^+$, $\varphi_n'$, $\psi_n'$ and obtain
\begin{equation}\label{ab}
	(\varphi_n'-\psi_n')(u_n''-u_n) = \bigo{n^{-m/2-\zeta}}.
\end{equation}
\begin{theorem}
	\label{T:convacc}
	Let the approximations $u_n'$ of tails $u_n$ satisfy \eqref{E:delta_prim} with $m$ such that
	\begin{equation}
		\label{E:m:zalozenie}
		m\geq	\begin{cases}
					0, & \text{if $\text{CF}\in\mDe_{20}\cup\mDn_{20}$,}\\
					1, & \text{if $\text{CF}\in\mDe_{10}$,}\\
					2, & \text{otherwise.}
				\end{cases}
	\end{equation}
	Let $u_n^+$ be given by \eqref{E:u_n^+}, and let auxiliary functions $\psi_n'$, $\varphi_n'$ be given by
	\[
		\psi_n' \coloneqq \frac{a_n' a_{n+1}}{\left(a_{n+1}+b_n'b_{n+1}+b_n'u_{n+1}'\right)^2},\qquad
		\varphi_n' \coloneqq 1+\frac m{2n}.
	\]
	Then the approximations $u_n''$ defined by
	\begin{equation}
		\label{E:u_n''}
		u_n'' \coloneqq
		\frac{\varphi_n'\,u_n^+ - \psi_n'\,u_n'}{\varphi_n'-\psi_n'}
	\end{equation}
	satisfy
	\begin{equation}\label{delta''}
		u_n''-u_n = \bigo{n^{-m/2-\theta}},
	\end{equation}
	where 
	\begin{equation}
	\label{theta}
	\theta \coloneqq \begin{cases}
				1, & \text{if $\text{CF}\in\mDe_{10}\cup\mDe_{20}$,}\\
				2, & \text{if $\text{CF}\in\mDn_{10}\cup\mDn_{20}\cup\wmD_{21}$,}\\
				4, & \text{if $\text{CF}\in\mD_{11}$.}\\
			\end{cases}
	\end{equation}
\end{theorem}
\begin{proof}
	Roughly speaking, the idea of the proof is to modify the relation~\eqref{E:delta_n^+} in order to obtain relation~\eqref{relation}.
	
	First, we use the following formula relating $\delta_{n+1}'$ and $\delta_n'$:
	\begin{equation}\label{deltan+1}
		\delta_{n+1}' = (1-\tfrac{m}{2n})\,\delta_n' + \bigo{n^{-m/2-\xi}},\qquad
		\text{where }
		\xi\coloneqq
		\begin{dcases}
		\tfrac{3}{2},    &\text{if $\text{CF}\in\mDe_{10}$,}\\
		2,               &\text{otherwise.}
		\end{dcases}
	\end{equation}
	This can be verified by substituting $n+1$ in place of $n$ in~\eqref{E:delta_prim} and using simple algebra.
	Observe that number $\xi$ is common for all classes in~$\mD$ except the class $\mDe_{10}$; it follows from Lemmata~\ref{ld1020} and~\ref{ld20}. In the case of $\mDe_{10}$ the formula~\eqref{deltan+1} is exactly the same as the one given in~\cite[Eq.~(3.7)]{Nowak06}.

	Second, let us observe some properties of function $\psi_n'$. One can check
	that
	\begin{equation}\label{E:rho}
		\psi_n' = c\,n^{-\rho}+\bigo{n^{-\rho-1/2}} \qquad \text{for some $c\neq 0$ and }
		\rho \coloneqq
		\begin{dcases}
			2,  & \text{in the case of $\mD_{11}$},\\
			0,  & \text{in all other cases.}
		\end{dcases}
	\end{equation}
	This is not obvious only in the case of $\wmD_{21}$. Clearly, the
coefficient of $n^4$ in the denominator of $\psi_n'$ is then equal to \(
%  		\label{E:D21:mianownik:nie:zero}
		p_{-2} + q_{-1}'\,q_{-1} + q_{-1}'\,\tau_{-2}
	\)
	and it does not vanish, because of the choice of $\tau_{-2}$, given in Theorem~\ref{tmu}.
	Next, using the definition~\eqref{E:psi_n} yields
	\[
	\psi_n - \psi_n' = 
		\frac{a_n'\,b_n'\,a_{n+1}}
		{
			\left(a_{n+1}+b'_n\,b_{n+1}+b'_n\,u_{n+1}\right)
			\left(a_{n+1}+b'_n\,b_{n+1}+b'_n\,u_{n+1}'\right)^2}
		\,\delta_{n+1}',
	\]
	and thus
	\begin{gather*}
		\psi_n-\psi_n' = c\,n^{-\frac m2-\rho'}+\bigo{n^{-\frac m2-\rho'-1/2}}\qquad (c\neq 0),
	\shortintertext{where}
		\rho'\coloneqq
		\begin{dcases*}
			1, & in the cases of $\mDe_{10},\,\mDn_{10}$ and $\wmD_{21}$,\\
			3, & in the case of $\mD_{11}$,\\
			2, & in the cases of $\mDe_{20}$ and $\mDn_{20}$.
		\end{dcases*}
	\end{gather*}
	Now, using the fact that $\tfrac{m}{2}+\xi+\rho\leq m+\rho'$ (cf.~\eqref{E:m:zalozenie}) and multiplying equation~\eqref{E:delta_n^+} by $\varphi_n'$ (some simple algebra is needed) yields the relation~\eqref{relation} for $\zeta \coloneqq \xi+\rho>0$.
	Further, let us observe that approximations $u_n''$, given
in~\eqref{E:u_n''}, can be found as a~solution of~the following equation,
derived from~\eqref{relation} by replacing right-hand side by zero and $u_n$
(stemmed by $\delta_n'$ and $\delta_n^+$) by $u_n''$:
	\[
		\varphi_n'(u_n^+-u_n'') - \psi_n'(u_n'-u_n'') = 0.
	\]
	Observe that equation~\eqref{ab} follows immediately by subtracting \eqref{relation} from the above formula.
	To complete the proof, let us state the following result:
	\begin{equation}
 		\label{E:epsilon_n}
	\varphi_n'-\psi_n' = c \, n^{-\eta} + \bigO{n^{-\eta-1/2}},\quad
	\text{where $c\neq0$ and }
		\eta\coloneqq	\begin{cases}
							\tfrac 12, & \text{in the case of $\mDe_{10}$,}\\
							1, & \text{in the case of $\mDe_{20}$,}\\
							0, & \text{in all other cases.}
						\end{cases}
	\end{equation}
	Hence, the formulae~\eqref{delta''} and~\eqref{theta} hold.	
	Observe that it is important that $c\neq0$ in~\eqref{E:epsilon_n}, since we need to divide the equation~\eqref{ab} by $\varphi_n'-\psi_n'$.
	
	To verify the formula~\eqref{E:epsilon_n}, let us determine the beginning coefficients, say $\sigma_j$, of the asymptotic expansion of $\varphi_n'-\psi_n'$:
	\[
		\varphi_n'-\psi_n' = \sigma_0 + \sigma_1 n^{-1/2} + \sigma_2 n^{-1} + \ldots.
	\]
	By Lemmata \ref{ld1020} and \ref{ld20}, we know that the coefficients $\sigma_{2j+1}$ $(j\geq0)$ may be nonzero only in the case of $\mDe_{10}$. The proof is based on the following facts:
	\begin{itemize}
		\item in the case of $\mDe_{10}$:\quad
 		\(	\displaystyle
 			\sigma_0=0,\quad \sigma_1=\tfrac{2q_0'\,\tau_{-1}}{p_{-1}}\ne 0;% =\frac{2q_0}{\tau_{-1}} % SPRAWDZIĆ
 		\)
		\item in the case of $\mDn_{10}$:\quad
 		\(	\displaystyle
 			\sigma_0= 1-\tfrac{p_{-1}'}{p_{-1}}\ne 0;
 		\)
 		\item in the case of $\mD_{11}$:\quad
 		\(	\displaystyle
 			\sigma_0=1;
 		\)
 		\item in the case of $\mDe_{20}$:\quad
 		\(	\displaystyle
 			\sigma_0=\sigma_1=0,\;\sigma_2=\tfrac{2\tau_{-2}\,q_0'+p_{-1}-p_{-1}'+\frac{m}{2}p_{-2}}{p_{-2}}\ne 0
 		\)
 		(cf.~\eqref{E:j/2:nierownosc:1});
		\item in the case of $\mDn_{20}$:\quad
 		\(	\displaystyle
 			\sigma_0=1-\tfrac{p_{-2}'}{p_{-2}}\ne 0;
 		\)
		\item in the case of $\wmD_{21}$:\quad
 		\(	\displaystyle
 			\sigma_0 = 1 - \tfrac{p_{-2}\,p_{-2}'}{\left(p_{-2}+q_{-1}\,q_{-1}'+q_{-1}'\,\tau_{-2}\right)^2}\ne 0.
 		\)
	\end{itemize}
	Only the last one needs some explanation. Namely, let us put
	\[
		z \coloneqq  \left(p_{-2}+q_{-1}\,q_{-1}'+q_{-1}'\,\tau_{-2}\right)^2 - p_{-2}\,p_{-2}'.
	\]
	One can verify that
	\[
		(2z-\varDelta)^2=\varDelta(p_{-2}+p_{-2}'+q_{-1}\,q_{-1}')^2,
	\]
	where $\varDelta$ is discriminant of the quadratic equation given in Lemma~\ref{lmueq}.
	In the case of subclass $\wmD_{21}$, we have obviously $\varDelta\neq 0$
	(see Table~\ref{Tab:Rozwazane_podklasy}).
	Suppose that $\sigma_0=0$. Using the following simple algebra:
	\[
		4\,p_{-2}\,p_{-2}' = (p_{-2}+p_{-2}'+q_{-1}\,q_{-1}')^2-\varDelta=\left(\varDelta(p_{-2}+p_{-2}'+q_{-1}\,q_{-1}')^2-\varDelta^2\right)/\varDelta,
	\]
	yields $z=0$, and thus $(2z-\varDelta)^2=\varDelta^2$. This implies that $p_{-2}\,p_{-2}'=0$, which is a~contradiction.
\end{proof}

\section{Algorithm}\label{s:alg}
Now, we summarize all the previous results. We give here the algorithm for
convergence acceleration of CFs of the class $\mD$ (see~\eqref{D}). It is almost
\ftime{fully rational} in the sense that it makes only arithmetical operations.
Only in the case of CFs of subclass $\mDe_{10}$ the initial calculations need to
compute square roots. In the main loop of the algorithm, only arithmetical
operations are used, for all the subclasses in $\mD$.

The algorithm is divided into two parts. First, we calculate initial approximations $u_{n0}$ of tails $u_n$ of CF, such that $u_{n0}-u_n=\bigo{n^{-m/2}}$, where
	\begin{equation*}
		m=
		\begin{dcases*}
		0, & in the cases of $\mDe_{20}$ and $\mDn_{20}$,\\
		1, & in the case of $\mDe_{10}$,\\
		2, & in other cases
		\end{dcases*}
	\end{equation*}
(cf.~\eqref{E:m:zalozenie}).
There we use the results given in Section~\ref{astail}. The second and main part
of algorithm is the iteration which improves, in each step, the accuracy of
approximations $u_{nj}$ of the tails $u_n$. Clearly, using the Theorem~\ref{tmu}
repeatedly, we obtain the following relation for approximants $u_{nj}$:
\[ u_{nj}-u_n = \bigo{n^{-m/2-j\theta}}. \]

\begin{algorithm}\label{alg}
Consider the continued fraction~\eqref{etvcfrac} belonging to $\mD$ (see~\eqref{D}).
\begin{description}
\item[\textbf{Initial step.}]
	Let us put the initial approximants $u_{n0}$ as follows:
	\begin{enumerate}
	\renewcommand{\theenumi}{\alph{enumi}}\renewcommand{\labelenumi}{\theenumi)}
		\item if CF$\in\mDe_{10}$, then \[ u_{n0}\coloneqq \tau_{-1}\,n^{1/2}+\tau_0, \]
			where
			\begin{equation*}
				%\label{E:De10:tau_-1}
				\tau_{-1} \coloneqq \sgn\left(\Re\frac{q'_0\sqrt{\frac{q_0\,p_{-1}}{q'_0}}}{p_{-1}}\right) \sqrt{\frac{q_0\,p_{-1}}{q'_0}}, \qquad
% 				\label{E:De10:tau0}
				\tau_0    \coloneqq \frac{2p'_{0}-2q'_{0}\,q_{0}+p_{-1}-2p_{0}}{4q'_{0}};
			\end{equation*}
		\item if CF$\in\mDn_{10}$, then \[ u_{n0} \coloneqq \tau_{-2}\,n + \tau_0, \]
			where
			\begin{align*}
% 				\label{E:Dn10:tau_-2}
				\tau_{-2}&\coloneqq\begin{dcases*} 0, & if $\abs{p_{-1}'}<\abs{p_{-1}}$, \\ \frac{p'_{-1}-p_{-1}}{q_0'} & otherwise, \end{dcases*}
				\\
				\tau_0 &\coloneqq
				\begin{dcases*}
				\frac{p'_{-1} q_0}{p_{-1}-p_{-1}'}, & if $\abs{p_{-1}'}<\abs{p_{-1}}$,\\
				{\frac{p_{{-1}}q_{{0}}}{p'_{{-1}}-p_{{-1}}}}+{\frac{p_{{-1}}+p'_{{0}}-p_{{0}}}{q'_{{0}}}}+{\frac{\left(p_{{-1}}-p'_{{-1}}\right) q'_1}{{q_{0}'}^2}}
				& otherwise;
				\end{dcases*}
			\end{align*}
		\item if CF$\in\mD_{11}$, then
			\(\ds
				u_{n0} \coloneqq \frac{p'_{-1}}{q'_{-1}};
			\)
		\item if CF$\in\mDe_{20}$, then
			\[
				u_{n0} \coloneqq %\label{E:De20:tau_-2}
				\dfrac{-\beta+\sgn\left(\Re\frac{\sqrt{\beta^2-4\alpha\gamma}}{p_{-2}}\right)\sqrt{\beta^2-4\alpha\gamma}}{2\alpha}\, n,
			\]
			where
			\(
				\alpha = q'_0, \; \beta = p_{-1}-p'_{-1}-p_{-2}, \; \gamma = -q_0\, p_{-2};
			\)
		\item if CF$\in\mDn_{20}$, then
			\[
				u_{n0} \coloneqq \tau_{-4}\, n^2 + \tau_{-2}\, n,
			\]
			where
			\begin{align*}
				\tau_{-4}&=\begin{dcases} 0 & \text{if $\abs{p_{-2}'}<\abs{p_{-2}}$}, \\ \frac{p'_{-2}-p_{-2}}{q_0'} & \text{otherwise}, \end{dcases}\\
				\tau_{-2} &=
				\begin{dcases*} 0, & if $\abs{p_{-2}'}<\abs{p_{-2}}$,\\
				\frac{2 p_{-2}+p'_{-1}-p_{-1}}{q'_0} + \frac{q'_1(p_{-2}-p'_{-2})}{{q'_0}^2} & otherwise;
				\end{dcases*}
			\end{align*}
		\item if CF$\in\wmD_{21}$, then
			\[
				u_{n0} \coloneqq \tau_{-2}\, n + \tau_0,
			\]
			where
			\begin{equation*}
				\tau_{0}=\frac{p'_{{-2}}q_{{0}}+p'_{{-1}}q_{{-1}}
				-\left(q'_{{-1}}q_{{0}}+q'_{{0}}q_{{-1}}-p'_{{-1}}+p_{{-1}}-p'_{{-2}}\right)\tau_{{-2}}
				-\left(q'_{-1}+q'_{{0}}\right)\tau_{-2}^{2}
				}{2\,q'_{{-1}}\tau_{{-2}}+p_{{-2}}-p'_{{-2}}+q'_{{-1}}q_{{-1}}},
			\end{equation*}
			and~$\tau_{-2}$ is one of two roots $\{x_0, x_1\}$ of equation
			\begin{equation*}
% 				\label{E:alg:rownanie_kwadratowe_na_tau}
				q_{-1}'\,x^2+(p_{-2}-p'_{-2}+q_{-1}\,q'_{-1})\,x-q_{-1}\,p'_{-2}=0,
			\end{equation*}
			which yields less absolute value
			\begin{equation*}
				\abs{\frac{p_{-2}'}{q_{-1}'}-x_k},\qquad k={0,1}.
			\end{equation*}
	\end{enumerate}
\item[\textbf{Main loop.}]
	Compute the approximations $u_{nj}$ using the following formula:
	\begin{equation*}
		u_{n,j+1} \coloneqq
		\frac{\varphi_n^{(j)}\,u_{nj}^+ - \psi_n^{(j)}\,u_{nj}}{\varphi_n^{(j)}-\psi_n^{(j)}},
		\qquad n\geq 1, \; j\geq 0,
	\end{equation*}
	where
	\begin{align*}
		u_{nj}^+ &\coloneqq \cfrac{a_n'}{b_n'+\tfrac{a_{n+1}}{b_{n+1}+u_{n+1,j}}},\\
		\varphi_n^{(j)} &\coloneqq 1+\left(\frac{m}{2}+j\theta\right) n^{-1},\\
		\psi_n^{(j)} &\coloneqq \frac{a'_n a_{n+1}}{\left(a_{n+1}+b'_n\,b_{n+1}+b'_n\,u_{n+1,j}\right)^2}.
	\end{align*}
\end{description}
\end{algorithm}

\section{Numerical results}\label{s:num}
We give examples of some remarkable mathematical constants and special functions
having continued fraction expansions. We apply our method (see
Algorithm~\ref{alg}) in order to approximate their value.

Consider a~two-variant CF of the form~\eqref{etvcfrac} belonging to the class $\mD$.
Of course, one may always approximate its value using the following classic methods:
\begin{enumerate}
\item computing classical approximants $S_n(0)$ for sufficiently large $n$,
\item applying the methods given in~\cite{Lorentzen95} or \cite{Nowak06} to even or odd part of CF,
\item using equivalent transformation in order to obtain limit $2$-periodic CF,
and using the techniques given in~\cite{Lorentzen95},
\item using the recent methods of Paszkowski~\cite{Paszkowski03}.
\end{enumerate}
See \cite[\S5]{Lorentzen08} for more details on numerical computation of CFs.
The first idea should be rather the last resort in order to compute any CF.
Convergence of CF from the class $\mD$ may be very slow (see
Example~\ref{ex1}).

The second method requires to derive even or odd part of CF (see, e.g.,
\cite[Thm. 2.19--2.20, pp. 86--87]{Lorentzen08}).
After that, in order to apply the techniques summarized by Lorentzen
in~\cite{Lorentzen95}, some equivalent transformation (see, e.g.,
\cite[Thm.~2.14, p.~77]{Lorentzen08}) into limit periodic CF may be needed.
However, for some subclasses in $\mD$, this process leads to the case that the methods given in~\cite{Lorentzen95} or \cite{Nowak06} are not applicable. For example, the `improvement machine', mentioned in Section~\ref{s:method}, can be applied only to CFs of the so-called \ftime{loxodromic} type; see, e.g., \cite[p.~73]{Lorentzen95} for definition of CFs of \ftime{loxodromic}, \ftime{elliptic}, \ftime{parabolic} and \ftime{identity} type.
Clearly, one can verify that odd (even) part of CF of the subclasses~$\mDe_{10}, \mDe_{20}$ is equivalent to a~limit periodic CF of parabolic type. Then the question of its convergence acceleration is much more subtle unless it is divergent. 
In the case of limit periodic continued fraction $\K(c_n/d_n)$ of parabolic
type, one can always try to use so-called \ftime{fixed point} method (see
\cite{Lorentzen95}, \cite[p.~219]{Lorentzen08}), i.e. computing modified
approximants $S_n(-1/2)$, provided that, without loss of generality, $c_n\to
-1/4$ and $d_n\to 1$. But, in order to accelerate its convergence, there are
some additional conditions required (see, e.g., \cite[Eq.~(4.10)]{Lorentzen95}).
In examples~\ref{ex2} and \ref{ex3} we show a~class of CFs (belonging to
$\mDe_{20}$ and~$\mDe_{10}$) having odd (even) part of parabolic type, for which
`fixed point' method does not work.

The third method is quite similar to the previous one, since it depends on the type of limit periodic CF.
Still the `improvement machine' can be applied only in the case of CF of loxodromic type. Other cases  need  more careful handling.
Anyway, all mentioned methods of computing limit periodic CFs do not cover the considered class $\mD$ of CFs.

The fourth method, established by Paszkowski~\cite{Paszkowski03}, is based on analytical transformation of tails of CF. It can also be applied only to some subclasses in $\mD$. For example, it does not handle the case of CF with tails having the asymptotic expansion in powers of $n^{-1/2}$, and so it cannot be applied to the subclass~$\mDe_{10}$. Moreover, it is quite difficult to compare such analytical method with Algorithm~\ref{alg}, which is rather of numerical character.

In the sequel, we give a~series of examples showing the application of our
method (Algorithm~\ref{alg}) in order to accelerate the convergence of CFs
belonging to the class $\mD$. We give additional remarks on computing their
value by using aforementioned methods, and, if it is possible, we show
a~comparison.

Before we begin, let us fix some notation to measure the \ftime{accuracy} of the method.
Consider the CF $\K(a_n/b_n; a_n'/b_n')$ of the class $\mD$, convergent to
the value $V$.
Let $\acc(x)$ denote number of exact significant decimal digits of $x$ in $V$:
\[
	\acc(x) \coloneqq -\log_{10}\abs{1-\frac{x}{V}}.
\]
We use the following notation:
\begin{equation*}%\label{E:d_nj}
	\delta_n \coloneqq \acc(S_n(0)), \qquad \delta_{nj} \coloneqq \acc\left(S_{2n-1}(u_{nj})\right),
\end{equation*}
where $u_{nj}$ are the values computed by Algorithm~\ref{alg}.
Let us remark that $u_{nj}$ is an~approximation of the tail $u_n=f^{(2n-1)}$ of
CF, and that is why we analyse the modified approximant $S_{2n-1}(u_{nj})$
(cf.~\eqref{ematail}). We call the number $\acc(S_{2n-1}(u_{nj}))$ accuracy of
$u_{nj}$. In all numerical examples, we observed that it is increasing in $j$,
unless it reaches the peak level, which can be, in some cases, more or less the
half of the arithmetic precision in a~given computer algebra system. We used
the system \textsf{Maple 14} with $128$-digit (decimal) arithmetic precision.

\begin{example}\label{ex1}
Consider the following CF expansion
\begin{equation}\label{E:33,12}
	\dfrac{4}{\psi\left(\frac{x+3+\nu}{4}\right)+\psi\left(\frac{x+3-\nu}{4}\right)-\psi\left(\frac{x+1+\nu}{4}\right)-\psi\left(\frac{x+1-\nu}{4}\right)}= x + \cftwosum{n=1}{\infty}{\polter{(2n-1)^2-\nu^2}{x}+\polter{(2n)^2}{x}}
\end{equation}
($\Re x>0$, cf.~\cite[p.~33, Eq.~(12)]{Perron}), where $\psi(x)$ denotes the
digamma function, $\psi(x) = \Gamma'(x)/\Gamma(x) $.
Let us put $x\coloneqq1$ and $\nu\coloneqq 1/2$. We have the following CF of subclass $\mDe_{20}$:
\begin{equation}
	\label{E:33,12:1:1/2}
	V \coloneqq 1.327052799890558739735... = %179836991513624862560402536001367798956905
	1+\cftwosum{n=1}{\infty}{\polter{(2n-1)^2-\frac14}{1}+\polter{(2n)^2}{1}}.
\end{equation}
It is very slowly convergent. For example, the $100$th classical approximant $S_{100}(0)=1.319558\ldots$ yields $\delta_{100}=2.25$. One can check that $\delta_{1000}=3.24$ and $\delta_{10000}=4.24$, so it seems that accuracy of classical approximants is increasing logarithmically. Nevertheless, one can derive very good approximation of its value by using modified approximants $S_n(\omega_n)$ with $\omega_n\approx f^{(n)}$. The Algorithm~\ref{alg} provides the approximations $u_{nj}$ of tails $f^{(2n-1)}$ of increasing accuracy in $j$, in the sense that
\begin{align*}
	u_{n0} &= \tau_{-2}\,n,\\
	u_{n1} &= \tau_{-2}\,n+\tau_{0}+{\frac{225}{512}}\,{n}^{-1}+\Order{-2},\\
	u_{n2} &= \tau_{-2}\,n+\tau_{0}+\tau_{2}\,{n}^{-1}+{\frac {7125}{32768}}\,{n}^{-2}-{\frac {868725}{1048576}}\,{n}^{-3}+\Order{-4},\\
	u_{n3} &= \tau_{-2}\,n+\tau_{0}+\tau_{2}\,{n}^{-1}+\tau_{4}\,{n}^{-2}+{\frac {497625}{4194304}}\,{n}^{-3}+\Order{-4},\\[-.6em]
\vdots\\[-.6em]
	u_{nj} &= u_n+\Order{-j},\qquad j\geq 4,
\end{align*}%
where $\tau_{i}$ are the coefficients of the asymptotic expansion of the tails
$u_n=f^{(2n-1)}$ (see~\eqref{eunexp}).
Let us observe that in each step of iteration in Algorithm~\ref{alg}, we derive
only one (nontrivial) coefficient of this expansion; let us remark that
$\tau_{2i-1}$ vanishes for all $i\in\fN_{0}$. This is the worst case to
accelerate (cf.~\eqref{theta}).
However, the approximation $u_{1, 10}$ yields about ten decimal digits of value $V$ of fraction \eqref{E:33,12:1:1/2}:
\[
	S_{1}(u_{1,10}) \approx \underline{1.327052799}780862, %1.327052799890559
	\quad\text{and so} \quad \delta_{1,10}=10.08.
\]
Table~\ref{tab:33,12_1_1/2} shows the accuracies of all approximations $u_{nj}$ needed to be computed in order to obtain the value $u_{1,10}$. One can observe that accuracy of $u_{1,j}$ (see the second row of the table) is increasing linearly in $j$.
\begin{table}[htb]
	\caption{Values $\delta_{nj}$ derived by Algorithm~\ref{alg} applied to CF (\ref{E:33,12:1:1/2}).}
	\label{tab:33,12_1_1/2}
	\centering\tablesize
	\begin{tabular}{c|rrrrrrrrrrrrrrr}
	\multicolumn{1}{c|}{$n\backslash j$} & \multicolumn{1}{c}{0} & \multicolumn{1}{c}{1} & \multicolumn{1}{c}{2} & \multicolumn{1}{c}{3} & \multicolumn{1}{c}{4} & \multicolumn{1}{c}{5} & \multicolumn{1}{c}{6} & \multicolumn{1}{c}{7} & \multicolumn{1}{c}{8} & \multicolumn{1}{c}{9} & \multicolumn{1}{c}{10} \\\hline
	1  & $1.24$ & $2.40$ & $3.24$ & $4.04$ & $4.82$ & $5.62$ & $6.44$ & $7.29$ & $8.17$ & $9.10$ & $10.08$\\
	2  & $1.79$ & $3.03$ & $3.90$ & $4.72$ & $5.53$ & $6.36$ & $7.22$ & $8.10$ & $9.03$ & $10.01$\\
	3  & $2.13$ & $3.46$ & $4.38$ & $5.25$ & $6.11$ & $6.98$ & $7.87$ & $8.80$ & $9.78$\\
	4  & $2.38$ & $3.78$ & $4.76$ & $5.68$ & $6.58$ & $7.49$ & $8.43$ & $9.40$\\
	5  & $2.57$ & $4.04$ & $5.08$ & $6.04$ & $6.98$ & $7.94$ & $8.91$\\
	6  & $2.73$ & $4.25$ & $5.34$ & $6.34$ & $7.33$ & $8.32$\\
	7  & $2.86$ & $4.44$ & $5.57$ & $6.61$ & $7.64$\\
	8  & $2.98$ & $4.60$ & $5.77$ & $6.85$\\
	9  & $3.08$ & $4.74$ & $5.95$\\
	10  & $3.17$ & $4.87$\\
	11  & $3.25$
	\end{tabular}
\end{table}
Moreover, it is worth to analyse the Figure \ref{fig:ulamek_33,12_zbieznosc_w_kolejnych_iteracjach_metody}, since it shows the values $\delta_{nj}$ --- accuracies of $u_{nj}$ --- for large values of $n$. One can observe that each step of iteration in Algorithm~\ref{alg} is of great importance, since it is much more productive than using approximations $u_{nj}$ for large values of $n$, provided that $j$ is fixed.
\begin{figure}[htb]%FIXME: uwaga na bottom
\centering
\includegraphics[width=.65\textwidth]{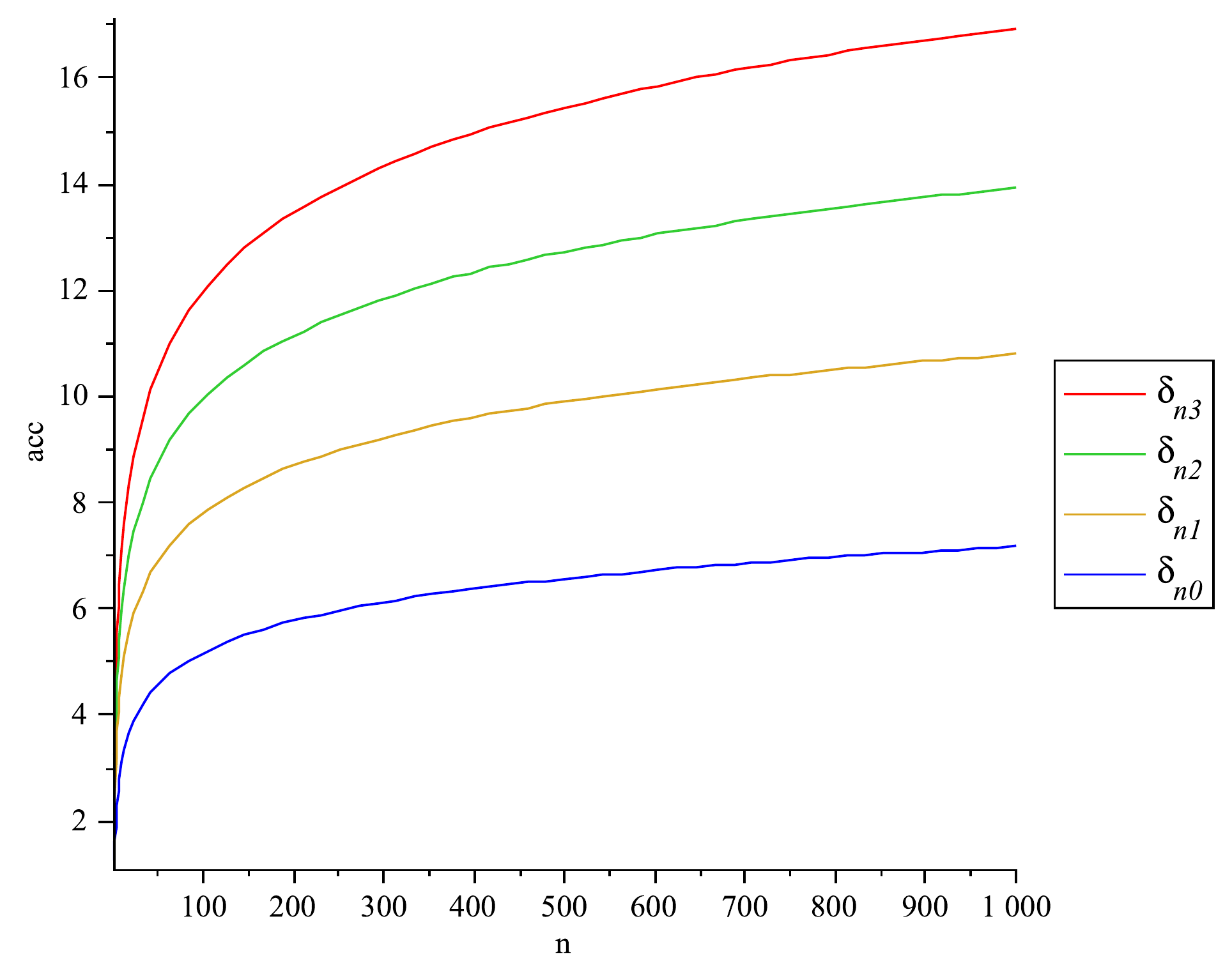}
\caption{Values $\delta_{nj}$ for CF~(\ref{E:33,12:1:1/2}).}
\label{fig:ulamek_33,12_zbieznosc_w_kolejnych_iteracjach_metody}
\end{figure}%

On the other hand, one can verify that even part of CF~\eqref{E:33,12} is equivalent to the CF of the form $\K(\tilde a_n/1)$,
where
\begin{equation}\label{E:asta}
	\tilde a_n+\frac14=\frac{1}{16}(x^2-1)\,n^{-2}+\frac{1}{32}(3\,x^{2}-3-\nu^2){n}^{-3}+\Order{-4},
\end{equation}
and so it is of parabolic type. Thus, we cannot use `improvement machine' of
Lorentzen and Waadeland (see \cite{Lorentzen95}, \cite[p.~227]{Lorentzen08}).
Let us remark that $\Re x>0$, so CF is convergent by some generalisation of
Worpitzky's theorem (see~\cite[pp.~45--50]{Wall48}, \cite[Thm.~3.30,
p.~136]{Lorentzen08}). Moreover, we get $a_n+\tfrac14=\bigo{n^{-3}}$, since we
took $x=1$.
Using the theorem of Thron and Waadeland \cite{ThronWaadeland80}, we show that
modified approximants $S_n(-1/2)$ are convergent to the value $V$ faster than
classical approximants $S_n(0)$, and so the `fixed point' method works. 
However, it is not very efficient, since the accuracy of $S_n(-1/2)$ behaves
similarly to $\delta_{n0}$ for consecutive values of $n$. Namely, for $n=10,
100, 1000, 10000$, we get the following values of $\acc(S_n(-1/2))$: $4.52,
6.36, 8.35, 10.34$, respectively.
In order to obtain $14$ decimal digits of $V$ we have to compute
$S_{670000}(-1/2)$. This should be compared with $\delta_{1,13}=14.0$ obtained
after only $13$ iterations in Algorithm~\ref{alg}.
\end{example}

\begin{example}\label{ex2}
Let us now set $x\coloneqq 1/2$, $\nu\coloneqq 1/2$ in~\eqref{E:33,12}. We
obtain CF with value $V\approx0.883414269615$ and even part being equivalent to
CF of the form $\K(\tilde a_n/1)$, where
\begin{equation}\label{E:asta1/2}
	\tilde a_n +\frac14 = -\frac{3}{64}\,n^{-2} + \Order{-3}.
\end{equation}
It is of parabolic type, too, hence we cannot use `improvement machine' of Lorentzen and Waadeland.
One can verify that conditions given by Thron and Waadeland in~\cite{ThronWaadeland80} are not satisfied, and we cannot use their results to accelerate the convergence, either. Indeed, `fixed point' method fails --- it gives only about $0.5$ accuracy more than classical approximants $S_n(0)$. For example, we have
$\delta_{10}=0.78,\;\acc(S_{10}(-1/2))=1.18$; $\delta_{100}=1.21,\;\acc(S_{100}(-1/2))=1.67$; $\delta_{1000}=1.69,\;\acc(S_{1000}(-1/2))=2.16$; $\delta_{10000}=2.19,\;\acc(S_{10000}(-1/2))=2.66$.
On the other hand, one can effectively approximate the value $V$ of the considered CF using Algorithm~\ref{alg}.
The accuracies $\delta_{1j}$ for $j=0,1,\ldots,10$ are:
\[
	1.02 , \quad 2.15 , \quad 3.00 , \quad 3.82 , \quad 4.65 , \quad 5.51 , \quad 6.41 , \quad 7.33 , \quad 8.29 , \quad 9.29 , \quad 10.32.
\]
One can verify that CF in~\eqref{E:33,12} is equivalent to a~the fraction considered by Paszkowski \cite{Paszkowski03}. Namely, we have
\begin{equation}\label{E:33,12:M}
	x+\cftwosum{n=1}{\infty}{\polter{(2n-1)^2-\nu^2}{x}+\polter{(2n)^2}{x}} \simeq x+\frac{4}{x}\,\mathcal{M}\left(\frac{x^2}{4}, -1, \frac{1-\nu^2}{4}, 0, 0\right),
\end{equation}
where class $\mathcal{M}$ is defined in~\cite[Eq.~(5.8)]{Paszkowski03}.
Using Paszkowski's analytical formulae gives a~faster convergent CF. For example, its classical approximant $\tilde f_{20}$ yields $\acc(\tilde f_{20})=21.6$, while our method (applied to~\eqref{E:33,12}) --- $\delta_{1,20}=20.0$.
\end{example}

\begin{example}\label{ex3}
Now, let us consider the following CF of subclass $\mDe_{10}$:
\begin{equation}
	\label{E:103,14}%EEEEEEEEEEEEEEEE
%	\tag{103,(14); 298,(13,16)}
	f(z,\alpha)\coloneqq 
	\left(z^{\alpha-1}e^z\int_{z}^{\infty} e^{-v} v^{-\alpha} dv\right)^{-1}
	\!\!\!\!\!
	=z+\cftwosum{n=1}{\infty}{\polter{n+\alpha-1}{1} + \polter{n}{z}},\quad
	z>0,\;\alpha\in\fR
\end{equation}
(cf.~\cite[p.~103, Eq.~(14)]{Perron}).
The asymptotic expansion of~its tails contains powers of~$n^{-1/2}$ (see~Theorem~\ref{tmu}), and thus we cannot use Paszkowski's method \cite{Paszkowski03} in order to accelerate its convergence.
Using its even part, we derive the following one-variant CF of subclass $\mC_{21}$:
\begin{equation}\label{E:103,14:C21}
	f(z,\alpha) = z+\polter{\alpha\,z}{z+1}+\cfinf[n=2]\polter{-(n-1)(n+\alpha-1)}{2n+z+\alpha-1}.
\end{equation}
Unfortunately, we cannot use the method of~\cite{Nowak06}, since the above fraction does not belong to the subclass $\mathcal{I}_1\subset\mC_{21}$ (see the notation introduced in \cite[Eq.~(2.3)]{Nowak06}).
Further, the equivalent transformation into limit periodic CF of the form $\K(\tilde a_n/1)$ gives
\begin{equation}\label{E:103,14,par,z}
	\tilde a_n = -\frac{1}{4}+\frac{z}{4}\,n^{-1}+\bigO{n^{-2}}.
\end{equation}
Let us put $z\coloneqq 1/16$ and~$\alpha\coloneqq 4$. Then CF
in~\eqref{E:103,14} is convergent to \[ V\approx
3.09147726049419952742569567195.\] We have $\delta_{10}=0.61$,
$\delta_{50}=3.02$, $\delta_{100}=4.45$ which shows quite slow convergence of
the fraction. Applying `fixed point' method to CF $\K(\tilde a_n/1)$ equivalent
to~\eqref{E:103,14:C21} seems to be worthless, since:
\begin{alignat*}{2}
	\delta_{10} &= 1.45,& \qquad \acc(S_{10}(-\tfrac12))&=1.45, \\
	\delta_{50} &= 4.45,& \acc(S_{50}(-\tfrac12))&=4.53, \\
	\delta_{100}&= 6.17,& \acc(S_{100}(-\tfrac12))&=6.23.
\end{alignat*}
Indeed, equation~\eqref{E:103,14,par,z} does not follow the assumptions of Thron and Waadeland~\cite{ThronWaadeland80}.
Let us remark that the computations of Algorithm~\ref{alg} is organized so that array $u_{nj}$ is upper triangular.
Thus, if we want to compute the approximation $u_{1,r}$, then we need to start
with $r+1$ initial approximations $u_{n0}$ $(n=1,2,\ldots,r+1)$.
Applying our method to \eqref{E:103,14} yields $u_{1,79}$ with accuracy
$\delta_{1,79}=26.23$. Let us note that $S_{n}(0)$ for $n=0,1,\ldots,160$ yield
at most $6$ exact significant decimal digits of $V$.
\end{example}

\begin{example}\label{ex4}
Now, let us consider CF~of subclass $\wmD_{21}$:
\begin{equation}\label{E:152,7}
	\frac{x}{\log(1+x)} = 1+\cftwosum{n=1}{\infty}{\polter{n^2 x}{2n}+\polter{n^2 x}{2n+1}},\qquad x\notin \left(-\infty,-1\right]
\end{equation}
(cf.~\cite[p.~152, Eq.~(7)]{Perron}).
We use its even part and obtain the CF of subclass $\mC_{21}$, of the form $\K(c_n/d_n)$ with
\begin{equation}\label{E:D21:cd}
\begin{split}
	c_n = -x^2\,n^2+2x^2\,n-\frac32\,x^2+\frac54\,x^2\,n^{-1}+\Order{-2},\\
	d_n = (2x+4)n-2-x+\left(\frac{3x}2+2\right)n^{-1}+\Order{-2}.
\end{split}
\end{equation}
One can easily check that it is of loxodromic type and belongs to subclass
$\mathcal{I}_1$ (see notation in~\cite[Eq.~(2.3)]{Nowak06}). Thus, one can use
either `improvement machine' or the method described in~\cite{Nowak06}. Of
course, one can also use `fixed point' method, classical approximants of CF, as
well. We try to compare their results with the results obtained by
Algorithm~\ref{alg} applied to CF~\eqref{E:152,7}.

Let us put complex value of parameter $x=-1.5+0.01\im$. Notice that it is
very close to the ray $(-\infty,-1]$, which is the divergence region
of~\eqref{E:152,7}. Hence, CF~\eqref{E:152,7} is convergent very slowly. Indeed,
we have
\[ \delta_{10}=-0.3, \quad \delta_{50}=-0.5,\quad \delta_{100}=0.1,\quad \delta_{200}=0.6,\quad \delta_{300}=0.9. \]
We applied `improvement machine' and the method described in~\cite{Nowak06} to
CF $\K(c_n/d_n)$ (cf.~\eqref{E:D21:cd}), and the Algorithm~\ref{alg} to
CF~\eqref{E:152,7}. All of the methods construct upper triangular array of
tails' approximations. We provide all of them with only fifteen initial
approximations. Table~\ref{tab:152,7:complex} presents the obtained accuracies.
One can observe that `improvement machine' yields only $4$ decimal digits, while method of~\cite{Nowak06} and Algorithm~\ref{alg} --- $10$ and $16$ exact significant decimal digits of $V$, respectively.
	\begin{table}[b]
		\caption{Computation of (\ref{E:152,7}) with $x=-1.5+0.01\im$ by the
		methods: `improvement machine' \cite[p.~227]{Lorentzen08},
		method of~\cite{Nowak06}, and Algorithm~\ref{alg}.}
		\label{tab:152,7:complex}
		\centering\tablesize
		\subtable[`Improvement machine': values of $\acc(S_n(w_{n}^{(j)}))$.]{
		\begin{tabular}{c|rrrrrrrrrrrrrrr}
		\multicolumn{1}{c|}{$n\backslash j$}  & \multicolumn{1}{c}{0} & \multicolumn{1}{c}{1} & \multicolumn{1}{c}{2} & \multicolumn{1}{c}{3} & \multicolumn{1}{c}{4} & \multicolumn{1}{c}{5} & \multicolumn{1}{c}{6} & \multicolumn{1}{c}{7} & \multicolumn{1}{c}{8} & \multicolumn{1}{c}{9} & \multicolumn{1}{c}{10} & \multicolumn{1}{c}{11} & \multicolumn{1}{c}{12} & \multicolumn{1}{c}{13} & \multicolumn{1}{c}{14} 
		\\\hline
		$ 1$ & $0.6$ & $0.8$ & $1.1$ & $1.4$ & $1.6$ & $1.9$ & $2.1$ & $2.4$ & $2.6$ & $2.8$ & $3.1$ & $3.3$ & $3.6$ & $3.8$ & $4.0$ \\
$ 2$ & $1.4$ & $1.7$ & $2.1$ & $2.4$ & $2.7$ & $3.0$ & $3.3$ & $3.6$ & $3.9$ & $4.2$ & $4.5$ & $4.7$ & $5.0$ & $5.3$ \\
$ 3$ & $1.8$ & $2.3$ & $2.7$ & $3.1$ & $3.5$ & $3.8$ & $4.2$ & $4.5$ & $4.8$ & $5.1$ & $5.4$ & $5.8$ & $6.1$ \\
$ 4$ & $2.0$ & $2.6$ & $3.1$ & $3.6$ & $4.0$ & $4.4$ & $4.8$ & $5.2$ & $5.6$ & $5.9$ & $6.2$ & $6.6$ \\
$ 5$ & $2.2$ & $2.9$ & $3.5$ & $4.0$ & $4.5$ & $4.9$ & $5.4$ & $5.8$ & $6.2$ & $6.6$ & $6.9$ \\
$ 6$ & $2.4$ & $3.1$ & $3.8$ & $4.3$ & $4.9$ & $5.4$ & $5.8$ & $6.3$ & $6.7$ & $7.1$ \\
$ 7$ & $2.5$ & $3.3$ & $4.0$ & $4.6$ & $5.2$ & $5.7$ & $6.2$ & $6.7$ & $7.2$ \\
$ 8$ & $2.7$ & $3.5$ & $4.2$ & $4.9$ & $5.5$ & $6.0$ & $6.6$ & $7.1$ \\
$ 9$ & $2.8$ & $3.6$ & $4.4$ & $5.1$ & $5.7$ & $6.3$ & $6.9$ \\
$10$ & $2.9$ & $3.8$ & $4.6$ & $5.3$ & $6.0$ & $6.6$ \\
$11$ & $3.0$ & $3.9$ & $4.7$ & $5.5$ & $6.2$ \\
$12$ & $3.0$ & $4.0$ & $4.9$ & $5.7$ \\
$13$ & $3.1$ & $4.1$ & $5.0$ \\
$14$ & $3.2$ & $4.2$ \\
$15$ & $3.2$
		\end{tabular}}\\
		\subtable[Method of~\cite{Nowak06}: values of $\acc(S_n(t_{nj}))$.]{
		\begin{tabular}{c|rrrrrrrrrrrrrrr}
		\multicolumn{1}{c|}{$n\backslash j$}  & \multicolumn{1}{c}{0} & \multicolumn{1}{c}{1} & \multicolumn{1}{c}{2} & \multicolumn{1}{c}{3} & \multicolumn{1}{c}{4} & \multicolumn{1}{c}{5} & \multicolumn{1}{c}{6} & \multicolumn{1}{c}{7} & \multicolumn{1}{c}{8} & \multicolumn{1}{c}{9} & \multicolumn{1}{c}{10} & \multicolumn{1}{c}{11} & \multicolumn{1}{c}{12} & \multicolumn{1}{c}{13} & \multicolumn{1}{c}{14} 
		\\\hline
		$ 1$ & $0.6$ & $1.8$ & $2.5$ & $3.3$ & $4.0$ & $4.7$ & $5.4$ & $6.1$ & $6.8$ & $7.5$ & $8.1$ & $8.8$ & $9.4$ & $10.1$ & $10.7$ \\
$ 2$ & $1.4$ & $2.4$ & $3.1$ & $3.9$ & $4.6$ & $5.3$ & $6.0$ & $6.7$ & $7.4$ & $8.1$ & $8.7$ & $9.4$ & $10.0$ & $10.7$ \\
$ 3$ & $1.8$ & $2.8$ & $3.6$ & $4.4$ & $5.1$ & $5.9$ & $6.6$ & $7.3$ & $7.9$ & $8.6$ & $9.3$ & $9.9$ & $10.6$ \\
$ 4$ & $2.0$ & $3.1$ & $4.0$ & $4.8$ & $5.6$ & $6.3$ & $7.0$ & $7.7$ & $8.4$ & $9.1$ & $9.8$ & $10.5$ \\
$ 5$ & $2.2$ & $3.3$ & $4.3$ & $5.1$ & $5.9$ & $6.7$ & $7.4$ & $8.1$ & $8.9$ & $9.6$ & $10.2$ \\
$ 6$ & $2.4$ & $3.5$ & $4.5$ & $5.4$ & $6.2$ & $7.0$ & $7.8$ & $8.5$ & $9.2$ & $10.0$ \\
$ 7$ & $2.5$ & $3.7$ & $4.8$ & $5.7$ & $6.5$ & $7.3$ & $8.1$ & $8.9$ & $9.6$ \\
$ 8$ & $2.7$ & $3.9$ & $5.0$ & $5.9$ & $6.8$ & $7.6$ & $8.4$ & $9.2$ \\
$ 9$ & $2.8$ & $4.1$ & $5.1$ & $6.1$ & $7.0$ & $7.9$ & $8.7$ \\
$10$ & $2.9$ & $4.2$ & $5.3$ & $6.3$ & $7.2$ & $8.1$ \\
$11$ & $3.0$ & $4.3$ & $5.5$ & $6.5$ & $7.4$ \\
$12$ & $3.0$ & $4.4$ & $5.6$ & $6.7$ \\
$13$ & $3.1$ & $4.5$ & $5.7$ \\
$14$ & $3.2$ & $4.6$ \\
$15$ & $3.2$
		\end{tabular}}
		\\
		\subtable[Algorithm~\ref{alg}: values of $\delta_{nj}$.]{%
		\begin{tabular}{c|rrrrrrrrrrrrrrr}
		\multicolumn{1}{c|}{$n\backslash j$}  & \multicolumn{1}{c}{0} & \multicolumn{1}{c}{1} & \multicolumn{1}{c}{2} & \multicolumn{1}{c}{3} & \multicolumn{1}{c}{4} & \multicolumn{1}{c}{5} & \multicolumn{1}{c}{6} & \multicolumn{1}{c}{7} & \multicolumn{1}{c}{8} & \multicolumn{1}{c}{9} & \multicolumn{1}{c}{10} & \multicolumn{1}{c}{11} & \multicolumn{1}{c}{12} & \multicolumn{1}{c}{13} & \multicolumn{1}{c}{14} 
		\\\hline
		$ 1$ & $1.1$ & $2.4$ & $3.4$ & $5.0$ & $5.8$ & $6.9$ & $7.9$ & $9.0$ & $10.0$ & $11.1$ & $12.1$ & $13.1$ & $14.2$ & $15.2$ & $16.2$ \\
$ 2$ & $1.8$ & $3.1$ & $4.4$ & $5.8$ & $6.7$ & $7.9$ & $8.9$ & $10.0$ & $11.0$ & $12.1$ & $13.1$ & $14.1$ & $15.2$ & $16.2$ \\
$ 3$ & $2.2$ & $3.6$ & $5.2$ & $6.4$ & $7.6$ & $8.7$ & $9.8$ & $10.9$ & $11.9$ & $13.0$ & $14.0$ & $15.1$ & $16.1$ \\
$ 4$ & $2.5$ & $4.0$ & $5.8$ & $6.9$ & $8.3$ & $9.4$ & $10.6$ & $11.6$ & $12.8$ & $13.8$ & $14.9$ & $16.0$ \\
$ 5$ & $2.7$ & $4.3$ & $6.3$ & $7.4$ & $8.9$ & $10.0$ & $11.3$ & $12.4$ & $13.5$ & $14.6$ & $15.7$ \\
$ 6$ & $2.9$ & $4.6$ & $6.7$ & $7.9$ & $9.4$ & $10.6$ & $11.9$ & $13.1$ & $14.2$ & $15.4$ \\
$ 7$ & $3.0$ & $4.8$ & $7.1$ & $8.3$ & $9.9$ & $11.1$ & $12.4$ & $13.7$ & $14.9$ \\
$ 8$ & $3.2$ & $5.0$ & $7.4$ & $8.7$ & $10.3$ & $11.6$ & $13.0$ & $14.2$ \\
$ 9$ & $3.3$ & $5.2$ & $7.7$ & $9.0$ & $10.7$ & $12.1$ & $13.4$ \\
$10$ & $3.4$ & $5.4$ & $7.9$ & $9.3$ & $11.0$ & $12.5$ \\
$11$ & $3.5$ & $5.6$ & $8.1$ & $9.6$ & $11.4$ \\
$12$ & $3.6$ & $5.7$ & $8.4$ & $9.9$ \\
$13$ & $3.7$ & $5.9$ & $8.5$ \\
$14$ & $3.7$ & $6.0$ \\
$15$ & $3.8$
		\end{tabular}}
	\end{table}
\end{example}

%%%%%%%%%%%%%%%%%%%%%%%%%%%%%%%%%%%
%%%%%%%%%%%%%%%%%%%%%%%%%%%%%%%%%%%
%%%                             %%%
%%%                             %%%
%%%       A P P E N D I X       %%%
%%%                             %%%
%%%                             %%%
%%%%%%%%%%%%%%%%%%%%%%%%%%%%%%%%%%%
%%%%%%%%%%%%%%%%%%%%%%%%%%%%%%%%%%%
% \newpage
%\begin{appendices}
\appendix
\renewcommand\thesection{Appendix~\Alph{section}}
\section{Computing coefficients of asymptotic expansion of tails}
\label{apa}
\renewcommand\thesection{\Alph{section}}
Here we show some formulae for coefficients $\tau_j$ of the expansion \eqref{eunexp} of tails $u_n$ for all considered subclasses of continued fractions (see~Table~\ref{Tab:Rozwazane_podklasy}).
All the proofs will basically follow the idea given by Wynn in \cite{Wynn59}. More precisely, we put the expansions~\eqref{easab} and \eqref{eunexp} into the recurrence relation~\eqref{oddtailsrel} in order to obtain the coefficients $\tau_j$.
After some rearrangement of terms, this leads to the following equation:
\begin{multline}
 	\label{E:u_n_rec_expansion}
	0=\sum_{m=-8}^{\infty}\!\sum_{i=-1}^{\lfloor\frac m2\rfloor+3}\,\sum_{j=-1}^{\lfloor\frac m2\rfloor+2-i} \negthickspace\negthickspace q_i'\,q_j\,\tau_{m-2i-2j}\,n^{-m/2}
	+\!\!\!\sum_{m=-10}^{\infty}\sum_{i=-1}^{\lfloor\frac m2\rfloor+4}\,\sum_{j=-4}^{m+4-2i} \negthickspace\negthickspace q_i'\,\tau_j\,\omega_{m-2i-j}\,n^{-m/2}+{}\\
	+\!\!\!\sum_{m=-8}^{\infty} \sum_{i=-2}^{\lfloor\frac m2\rfloor+2} (p_i\,\tau_{m-2i}-p_i'\,\omega_{m-2i})\,n^{-m/2}
	-\!\!\!\sum_{m=-3}^{\infty}\,\sum_{i=-2}^{m+1} p_i'\,q_{m-i}\,n^{-m},
\end{multline}
where $\omega_j$ are coefficients of the expansion of \( u_{n+1} = \sum_{j=-4}^\infty \omega_j\, n^{-j/2}. \)
One can check that
\begin{align*}
	\omega_{-4} &= \tau_{-4},\quad \omega_{-3}=\tau_{-3},\quad \omega_{-2}=2\tau_{-4}+\tau_{-2},\quad \omega_0=\tau_{-4}+\tau_{-2}+\tau_0,\\
	\omega_{2j-1}&=\tau_{-3}\binom{3/2}{j+1}+\tau_{-1}\binom{1/2}{j}+\tau_{1}\binom{-1/2}{j-1}+\ldots+\tau_{2j-1}\binom{-(2j-1)/2}{0}, \quad j\in\fN\cup\{0\},\\
	\omega_{2j}&=\tau_{2}\binom{-1}{j-1}+\tau_4\binom{-2}{j-2}+\ldots+\tau_{2j}\binom{-j}{0}, \quad j\in\fN,
\end{align*}
where $\tbinom{z}{m}$ means \ftime{generalised Newton's symbol}
\[
	\binom{z}{m}=(-1)^m\frac{(-z)(-z+1)\cdots(-z+m-1)}{m!},\qquad z\in\fC,\,m\in\fN\cup\{0\}.
\]
Let $c_m$ denote the coefficients of powers of $n^{-m/2}$
in~\eqref{E:u_n_rec_expansion} ($m=-10, -9, \ldots$);
the first few of them are as follows:
\begin{subequations}\label{E:rwm}
\label{E:tau:equations}
\begin{align}
	\label{E:m=-10}
		c_{-10} = {} & q'_{{-1}}\,\tau_{-4}^{2},\\
	\label{E:m=-9}
		c_{-9} = {} & 2\,q'_{{-1}}\,\tau_{{-4}}\,\tau_{{-3}},\\
	\label{E:m=-8}
		\begin{split}
		c_{-8} = {} & q'_{{-1}}\,q_{{-1}}\,\tau_{{-4}}+q'_{{-1}}\,\tau_{{-4}}
		\left( 2\,\tau_{{-4}}+\tau_{{-2}} \right)+\\
		&+q'_{{-1}}\,\tau_{-3}^{2}+q'_{{-1}}\,\tau_{{-2}}\,\tau_{{-4}}+q'_{{0}}\,\tau_{-4}^{2}+\tau_{{-4}}(p_{{-2}}-p'_{{-2}}),
		\end{split}
	\\
	\label{E:m=-7}
		\begin{split}
		c_{-7} ={}& \tau_{{-3}}\,q'_{{-1}}\,q_{{-1}}+q'_{{-1}}\,\tau_{{-4}}
		\left( \tau_{{-1}}+3\tau_{{-3}}/2 \right) +q'_{{-1}}\,\tau_{{-3}}
		\left( 2\,\tau_{{-4}}+\tau_{{-2}} \right)+\\
		&+\tau_{{-2}}\,q'_{{-1}}\,\tau_{{-3}}+
		\tau_{{-1}}\,q'_{{-1}}\,\tau_{{-4}}+2\,\tau_{{-4}}\,q'_{{0}}\,\tau_{{-3}}+\tau_{{-3}}\,p_{{-2}}-p'_{{-2}}\,\tau_{{-3}}.
		\end{split}
\end{align}
\end{subequations}
Equating them to zero gives the infinite system of equations satisfied by the
unknown coefficients $\tau_j$:
\begin{equation}\label{E:uklad}
	c_m=0, \qquad m=-10,-9, \ldots,
\end{equation}
Taking into account the forms of numerators $a_n$ and~denominators $b_n$
(see~\eqref{easab}), there easily follows that $\mu\leq 2$
if~CF$\not\in\mD_{20}$. Hence, the first few equations of~\eqref{E:uklad} are
obviously satisfied.
In Lemma~\ref{ld1020} we analyse this in the case of all considered subclasses
of CFs.
In Lemma \ref{lmueq}, we show the form of quadratic equation satisfied by the
beginning coefficient $\tau_\mu$.

\begin{lemma}\label{ld1020}
The following formulae hold:
\begin{enumerate}
\item For $\text{CF}\in\mDe_{10}$:
	\[ \tau_{-2} = 0, \qquad \tau_0 = \frac{2p'_{0}-2q'_{0}\,q_{0}+p_{-1}-2p_{0}}{4q'_{0}}; \]
\item In the case of $\mDn_{10}$:
	\[ \tau_{-1} = \tau_{1} = \tau_3 = \ldots = 0, \qquad 
	\tau_0 =
			\begin{cases}
				\frac{p'_{-1} q_0}{p_{-1}-p_{-1}'} & \text{if $\tau_{-2}=0$},\\
				{\frac{p_{{-1}}q_{{0}}}{p'_{{-1}}-p_{{-1}}}}+{\frac{p_{{-1}}+p'_{{0}}-p_{{0}}}{q'_{{0}}}}+{\frac{\left(p_{{-1}}-p'_{{-1}}\right) q'_1}{{q_{0}'}^2}}
				& \text{if $\tau_{-2}\neq 0$};
			\end{cases} \]
\item For $\text{CF}\in\mD_{11}$:
	\[ \tau_{-1} = \tau_{1} = \tau_3 = \ldots = 0, \qquad 
	\tau_0 = 
		\begin{cases}
			\frac{p'_{-1}}{q'_{-1}} & \text{if $\tau_{-2}=0$},\\
			q_{-1}-q_0 - \frac{p_{-1}}{q'_{-1}} & \text{if $\tau_{-2}\ne 0$};
		 \end{cases} \]
\item For $\text{CF}\in\mDe_{20}$:
	\[ \tau_{-4} = \tau_{-3} = 0; \]
\item For $\text{CF}\in\mDn_{20}$:
	\[ \tau_{-3} = \tau_{-1} = \tau_1 = \ldots = 0, \qquad 
	\tau_{-2} = 
		\begin{cases}
			0 & \text{if $\tau_{-4}=0$},\\
			\frac{2 p_{-2}+p'_{-1}-p_{-1}}{q'_0} + \frac{q'_1(p_{-2}-p'_{-2})}{{q'_0}^2}
			& \text{if $\tau_{-4}\ne0$};
		\end{cases} \]
\item For $\text{CF}\in\wmD_{21}$:
	\begin{multline*} \tau_{-3} = \tau_{-1} = \tau_1 = \ldots = 0, \\
	\tau_{0}=
			\frac{
			 p'_{{-2}}q_{{0}}+p'_{{-1}}q_{{-1}}
			 -\left(q'_{{-1}}q_{{0}}+q'_{{0}}q_{{-1}}-p'_{{-1}}+p_{{-1}}-p'_{{-2}}\right)\tau_{{-2}}
			 -\left(q'_{-1}+q'_{{0}}\right)\tau_{-2}^{2}
			 }{2\,q'_{{-1}}\tau_{{-2}}+p_{{-2}}-p'_{{-2}}+q'_{{-1}}q_{{-1}}}.
	\end{multline*}
\end{enumerate}
\end{lemma}
\begin{proof}
Each part of the lemma has to be considered separately.
For the subclass $\mD_{10}$, the equations~\eqref{E:uklad} for
$m=-10,-9,\ldots,-5$ are trivially satisfied. For $m=-4$, we obtain that
	\begin{equation}
 		\label{E:D10:m=-4}
		0=\tau_{-2}\left(q_0'\,\tau_{-2}-p'_{-1}+p_{-1}\right),
	\end{equation}
so there must be $\tau_{-2}=0$ in the case of $\mDe_{10}$. Then the equation
of~\eqref{E:uklad} with $m=-3$ is evidently satisfied, and those
corresponding to $m=-2, -1$ can be simplified to the form:
\begin{gather}
 		\label{E:De10:m=-2}
		q_0'\,\tau_{-1}^2-q_0\, p_{-1}=0,\\
		\nonumber%\label{E:De10:m=-1}
		\tau_{-1}\left(2\,p'_0-4\,\tau_0 q'_0-2\,q'_0 q_0-2\,p_0+p_{-1}\right)=0.
\end{gather}
The latter equation yields the result of the part 1.

Using similar analysis, one can prove the part 4.
Namely, in the case of class $\mD_{20}$, the equations \eqref{E:uklad} with
$m=-10$ and~$m=-9$ are obviously satisfied.
However, equations corresponding to $m=-8$ and~$m=-7$ can be written in the form
	\begin{align}
		\label{E:D20:m=-8}
		\tau_{-4}\left(q'_{0}\,\tau_{-4}+p_{-2}-p'_{-2}\right) &= 0,\\
		\label{E:D20:m=-7}
		\tau_{-3}\left(2 q'_{0}\,\tau_{-4}+p_{-2}-p'_{-2}\right) &=0.
	\end{align}
Hence, $\tau_{-4}=0$ in the case of $\mDe_{20}$.
Then equation~\eqref{E:D20:m=-7} is trivially satisfied, and the next equation
of \eqref{E:uklad}, related to $m=-6$, can be written as
\[ q_0'\,\tau_{-3}^2 = 0. \]
This proves the result of the part 4.

The proofs of the remaining parts have a~lot in common, so that we
show only the proof of the part 2.
Equation of~\eqref{E:uklad}, related to $m=-3$, can be simplified to the form
\[
	\tau_{{-1}} \left( 2\,q'_{{0}}\tau_{{-2}}+p_{{-1}}-p'_{{-1}} \right)=0.
\]
Comparing this with \eqref{E:D10:m=-4} implies that $\tau_{-1}=0$, regardless of
whether $\tau_{-2}=0$ or not.
Formula for $\tau_0$ follows from the equation of the system~\eqref{E:uklad}
with $m=-2$. Next, using induction on $j\geq 0$, one verifies that
the equation of~\eqref{E:uklad}, related to $m=2j-3$, can be simplified to the
form
\[
	\tau_{{2j-1}} \left( 2\,q'_{{0}}\tau_{{-2}}+p_{{-1}}-p'_{{-1}} \right)=0.
\]
Hence, all the coefficients $\tau_{1}, \tau_{3}, \ldots$ vanish.
\end{proof}

\begin{lemma}\label{lmueq}
The beginning coefficient $\tau_{-\mu}$ of~expansion \eqref{eunexp} satisfies the quadratic equation
\begin{equation}
	\label{etaueq}
	\alpha\,\tau^2+\beta\,\tau+\gamma=0,
\end{equation}
where $\alpha,\beta,\gamma$ are given in Table~\ref{tab:alpha}.
\end{lemma}
\begin{proof}
The proof uses Lemma~\ref{ld1020} and the arguments in its proof.
In the cases $\mDn_{10}$ and~$\mDe_{10}$ the coefficients $\alpha, \beta, \gamma$ result from~\eqref{E:D10:m=-4} and~\eqref{E:De10:m=-2}, respectively.
Considering the subclass $\mD_{11}$, the equation~\eqref{E:uklad} for
$m=-10,-9,-8,-7$ (cf.~\eqref{E:m=-10}--\eqref{E:m=-7}) holds trivially, while
for $m=-6$ can be written in the form
\begin{equation*}
	%\label{E:D11:m=-6}
	\tau_{-2}\,q'_{-1}\left(\tau_{-2}+q_{-1}\right)=0.
\end{equation*}
%This completes the proof of the case $\mD_{11}$.
In the case of subclass $\mDe_{20}$, using Lemma~\ref{ld1020}, we have that
the equation of~\eqref{E:uklad} with $m=-5$ is trivially satisfied.
The next equation, related to $m=-4$, is of the form
\[
	q'_{0}\,\tau_{-2}^2+\left(p_{-1}-p'_{-1}-p_{-2}\right)\tau_{-2}-p_{-2}\,q_{0}=0.
\]
In the case of $\mDn_{20}$, the coefficients $\alpha, \beta, \gamma$
easily result from~\eqref{E:D20:m=-8}.
Finally, in the case of $\wmD_{21}$, the equations of the system~\eqref{E:uklad}
related to $m=-10,-9,-8,-7$ (cf.~\eqref{E:m=-10}--\eqref{E:m=-7}) are
obviously satisfied, and the one corresponding to $m=-6$ can be written in the
form
\begin{equation}
	\label{E:D21:m=-6}
	q'_{-1}\,\tau_{-2}^2+\left(p_{-2}-p'_{-2}+q'_{-1}\,q_{-1}\right)\tau_{-2}-p'_{-2}\,q_{-1} = 0.
\end{equation}
\end{proof}

The Lemma \ref{ld20} extends the results given in the part 4
of~Lemma~\ref{ld1020}, for the subclass $\mDe_{20}$. Notice that Lemma
\ref{ld20} follows by application of Theorem~\ref{tmu} which is proved using
Lemma \ref{ld1020}.
\begin{lemma}\label{ld20}
	In the case of $\mDe_{20}$, we have
	\begin{equation*} \tau_{-1} = \tau_{1} = \tau_3 = \ldots = 0, \qquad
			\tau_{0}=
			\frac{%
				p'_{-1}\,q_{0}+p_{-2}\,q_{1}%-\tau_{-1}^{2}\,q'_{0}
				-\left(q'_{1}+q'_{0}\right)\tau_{-2}^{2}+
				\left(p'_{-1}+p'_{0}-p_{0}-q'_{0}\,q_{0}\right)\tau_{-2}%
			}{%
				2\,\tau_{-2}q'_{0}+p_{-1}-p'_{-1}
			}.
	\end{equation*}
\end{lemma}
\begin{proof}
	We continue the strategy of the previous proof. The equations~of
	the system~\eqref{E:uklad} corresponding to $m=-3$ and $m=-2$ simplify to
	\begin{gather*}
		\tau_{-1}\left(2\tau_{-2}q'_{0}+p_{-1}-p'_{-1}-\frac12 p_{-2}\right) = 0,\\
		\shortintertext{and}
		\tau_{{0}}\left(2\,\tau_{{-2}}q'_{{0}}+p_{{-1}}-p'_{{-1}}\right)=
		p'_{{-1}}q_{{0}}+p_{{-2}}q_{{1}}-\tau_{-1}^{2}q'_{{0}}
		-\left(q'_{{1}}+q'_{{0}}\right)\tau_{-2}^{2}+
		\left(p'_{{-1}}+p'_{{0}}-p_{{0}}-q'_{0}q_{0}\right)\tau_{-2},
	\end{gather*}
	respectively.
	The latter equation yields the formula for $\tau_0$.
	One can check that the coefficient of $\tau_{m}$ in the equation of the
	system~\eqref{E:uklad} related to $m>-2$ is given by
	\begin{equation}\label{e:coeff}
		2\tau_{-2}q'_{0}+p_{-1}-p'_{-1}+\frac{m} 2 p_{-2}.
	\end{equation}
	Moreover, using the induction on $j\in\fN_0$, one may verify that the equation of~\eqref{E:uklad}, corresponding to $m=2j-1$, can be simplified to the form
	\begin{equation*}
		\tau_{2j-1}\left(2\tau_{-2}q'_{0}+p_{-1}-p'_{-1}+\frac{2j-1}2 p_{-2}\right) = 0.
	\end{equation*}
	To complete the proof, we show that expression \eqref{e:coeff}
	cannot vanish for any integer $m>-2$.
	Namely, suppose that
	\[ 2\tau_{-2}\,q'_{0}+p_{-1}-p'_{-1}+\frac{m} 2 p_{-2} = 0 \]
	for some $m\in\fN\cup\{-1,0\}$.
	Then we have
	\begin{equation}
		\label{E:j/2:nierownosc:1}
		\frac12 \geq -\frac{m}{2} = \frac{2\tau_{-2}q'_{0}+p_{-1}-p'_{-1}}{p_{-2}}.
	\end{equation}
	Let us denote
	\[
		\varkappa := \frac{\sqrt{\beta^2-4\alpha\,\gamma}}{p_{-2}},
	\]
	where $\alpha,\beta,\gamma$ are given in the Table~\ref{tab:alpha}
	(cf.~the class $\mDe_{20}$).
	From Theorem~\ref{tmu}, we derive that
	\[
		2\tau_{-2}\,q'_{0}+p_{-1}-p'_{-1} = p_{-2} + \sgn\left(\Re\varkappa\right)\sqrt{\beta^2-4\alpha\gamma}
	\]
	(cf.~\eqref{E:De20:tau_-2}), and hence the
	inequality~\eqref{E:j/2:nierownosc:1} can be written in the form
	\[
		\frac{1}{2} \geq -\frac{m}{2} = 1+\sgn\left(\Re\varkappa\right) \varkappa.
	\]
	We conclude that $\varkappa$ is a~real number satisfying $\tfrac12 \geq 1+\abs{\varkappa} \geq 1$,
	which is a~contradiction.
\end{proof}
%\end{appendices}
%ENDOFTEXT - do not remove this comment
\bibliographystyle{abbrv}
\bibliography{bibliography}
\end{document}